\documentclass[11pt, a4paper]{article}

\usepackage{inputenc} 
\usepackage[english]{babel}
\usepackage{dsfont}

\usepackage[totalheight=24 true cm, totalwidth=17 true cm]{geometry}
\usepackage{enumitem}
\usepackage[toc,page]{appendix}

\setlist[itemize]{noitemsep}
\setlist[enumerate]{noitemsep,label=\textup{(\roman*)}}

\usepackage{bbding}
\usepackage{amsmath,multirow}
\usepackage{amsthm}
\usepackage{amssymb}
\usepackage{subcaption}
\usepackage{mathrsfs}
\usepackage{mathtools}
\usepackage{ stmaryrd }
\usepackage{url}
\usepackage{wrapfig}
\usepackage{starfont}
\usepackage{pifont}
\usepackage{eurosym}

\usepackage{imakeidx}
\makeindex[]

\usepackage{multicol}
\usepackage{relsize}
\usepackage{float}
\usepackage{tikz,pgf}
\usepackage{tikz-cd}
\usepackage{wasysym}
\usepackage{graphicx}
\usepackage{fancyhdr}
\usepackage{verbatim}

\usepackage{pgfplots}

\pgfplotsset{compat=1.15}
\usepackage{mathrsfs}
\usetikzlibrary{arrows}

\newcommand{\liftmat}{\mathcal{M}_q^\gamma(M)}
\newcommand{\dimm}[2]{\dim_#1^#2}
\newcounter{cases}
\newcounter{subcases}
\newcounter{subsubcases}
\newenvironment{mycases}
  {%
    \setcounter{cases}{0}%
    \def\case
      {%
        \par\noindent
        \refstepcounter{cases}%
        \textbf{Case \thecases.}
      }%
  }
  {%
    \par
  }

\newcommand\restr[2]{{
  \left.\kern-\nulldelimiterspace 
  #1 
  \vphantom{\small|} 
  \right|_{#2} 
  }}
\newcommand{\CC}{\mathbb{C}}

\newcommand{\VMLIFT}{V_{M}^{\text{lift}}}
\newcommand{\rk}{\text{rk}}

\newcommand{\CCC}{\mathcal{C}}

\newcommand{\Lift}{\textup{Lift}}
\newcommand{\VCM}{V_{\mathcal{C}(M)}}
\newcommand{\VCN}{V_{\mathcal{C}(N)}}

\usepackage[final]{pdfpages}
\usepackage{eurosym}
\newtheorem{theorem}{Theorem}[section]

\newtheorem{lemma}[theorem]{Lemma}

\theoremstyle{definition}
\newtheorem{definition}[theorem]{Definition}
\newtheorem{proposition}[theorem]{Proposition}
\newtheorem{example}[theorem]{Example}

\newtheorem{notation}[theorem]{Notation}
\newtheorem{remark}[theorem]{Remark}

\usepackage[all,dvips,knot,web,arc,curve,color,frame]{xy}
\usepackage[all,knot,arc,color,web]{xy}
\xyoption{arc}
\xyoption{web}
\xyoption{curve}

\newtheorem{theoremA}{Theorem}

\newtheorem{theoremB}{Theorem}

\newtheorem{strategy}{Strategy}
\numberwithin{equation}{section}

\theoremstyle{plain}

\newtheorem{question}[theorem]{Question}
\usepackage[bookmarksnumbered=true]{hyperref} 
\hypersetup{
     colorlinks = true,
     linkcolor = blue,
     anchorcolor = blue,
     citecolor = teal,
     filecolor = blue,
     urlcolor = blue
     }
\everymath{\displaystyle}
\allowdisplaybreaks

\makeatletter
\def\mathcenterto#1#2{\mathclap{\phantom{#1}\mathclap{#2}}\phantom{#1}}
\let\old@widetilde\widetilde
\def\widetildeto#1#2{\mathcenterto{#2}{\old@widetilde{\mathcenterto{#1}{#2\,}}}}
\let\old@widehat\widehat
\def\widehatto#1#2{\mathcenterto{#2}{\old@widehat{\mathcenterto{#1}{#2\,}}}}
\makeatother

\newcommand{\ip}[2]{\langle  #1,#2 \rangle} 

\newcommand{\size}[1]{\left| #1 \right|} 
\newcommand{\pare}[1]{\left( #1 \right)} 
\newcommand{\set}[1]{{\left\{ #1 \right\}}} 

\newcommand{\corch}[1]{\left[ #1 \right]} 

\newcommand*\closure[1]{\overline{#1}}

\DeclareMathOperator{\rank}{rk}

\def\dim{\operatorname{dim}}

\DeclareMathOperator{\sspan}{span}

\hypersetup{
    colorlinks=true,
    linkcolor=blue,
    filecolor=magenta,      
    urlcolor=cyan,
    pdftitle={Overleaf Example},
    pdfpagemode=FullScreen,
    }
    
\setlist[itemize]{noitemsep}

\normalfont

\usepackage{bbm}
\usepackage{dsfont}

\usepackage{url}
\usepackage{amsmath,blkarray,booktabs, bigstrut}
\usepackage{tikz}
\usepackage{multirow}
\usepackage{xcolor}
\usepackage{color}
\title{
Algebraic Geometry of Cactus, Pascal, and Pappus Matroids}
\author{Emiliano Liwski, Fatemeh Mohammadi, and Lisa Vandebrouck}
\date{}

\begin{document}
\maketitle

\begin{abstract}
We study rank-three matroids, known as point-line configurations, and their associated matroid varieties, defined as the Zariski closures of their realization spaces. Our focus is on determining finite generating sets of defining equations for these varieties, up to radical, and describing the irreducible components of the corresponding circuit varieties. We generalize the notion of cactus graphs to matroids, introducing a family of point-line configurations whose underlying graphs are cacti.
Our analysis includes several classical matroids, such as the Pascal, Pappus, and cactus matroids, for which we provide explicit finite generating sets for their associated matroid ideals. The matroid ideal is the ideal of the matroid variety, whose construction involves a saturation step with respect to all the independence relations of the matroid. This step is computationally very expensive and has only been carried out for very small matroids. We provide a complete generating set of these ideals for the Pascal, Pappus, and cactus matroids. The proofs rely on classical geometric techniques, including liftability arguments and the Grassmann--Cayley algebra, which we use to construct so-called bracket polynomials in these ideals. In addition, we prove that every cactus matroid is realizable and that its matroid variety is irreducible.

\end{abstract}


\vspace{-2mm
}

\section{Introduction}


Matroids are a combinatorial abstraction of linear dependence among vectors~\cite{whitney1992abstract, Oxley, piff1970vector}. Given a finite set of vectors in a vector space, one can associate a matroid by considering the collection of their linearly dependent subsets. When this process is reversible and a given matroid $M$ arises from such a set, we say that the set is a realization of $M$. The space of all realizations of $M$ is denoted by $\Gamma_M$, and its Zariski closure defines the matroid variety $V_M$, which encodes rich geometric structure. Introduced in~\cite{gelfand1987combinatorial}, matroid varieties have since become a central object of study~\cite{Vakil, Sidman, liwski2024pavingmatroidsdefiningequations, sturmfels1989matroid, feher2012equivariant, knutson2013positroid}. In this work, we study the problem of computing a complete set of defining equations for $V_{M}$. 

\smallskip
The primary goal of this paper is to determine a complete set of defining equations for the matroid variety $V_M$. This problem is known to be difficult, as illustrated in~\cite{pfister2019primary}, where the authors developed an algorithm to compute the defining equations of the matroid variety associated with the $3\times 4$ grid configuration, which has $16$ circuits of size $3$. Using {\tt Singular}, they showed that the corresponding ideal is generated by $44$ polynomials. However, their approach pushed the capabilities of existing computer algebra systems and yielded components without an evident combinatorial interpretation. In contrast, the methods presented here apply to broader classes of matroids and produce defining equations with a clear combinatorial and geometric interpretation.


\smallskip
We now summarize the main contributions of the paper. Throughout, we fix a 
simple rank-three matroid $M$ on the ground set $[d]$. 
To provide context, we first recall the key varieties associated to~$M$. 

\begin{definition}\label{first definition}
Let $M$ be a matroid on $[d]$ of rank at most three.
\begin{itemize}
\item We say that a collection of vectors $\gamma=\{\gamma_{1},\ldots,\gamma_{d}\}\subset \CC^{3}$ is a realization of $M$ if it satisfies:
\[\{i_{1},\ldots,i_{p}\} \ \text{is a dependent set of $M$} \Longleftrightarrow \{\gamma_{i_{1}},\ldots,\gamma_{i_{p}}\}\ \text{is linearly dependent.}\]
The {\em realization space} of $M$ is defined as 
$\Gamma_{M}=\{\gamma\subset \CC^{3}: \gamma \ \text{is a realization of $M$}\}\subset \CC^{3d}$.
The {\em matroid variety} $V_M$ of $M$ is defined as $\closure{\Gamma_M}$. The 
ideal $I_{M}=I(V_{M})$ is called the matroid ideal.
\item We say that $\gamma=\{\gamma_{1},\ldots,\gamma_{d}\}\subset \CC^{3}$ includes the dependencies of $M$ if it satisfies:
\[\set{i_{1},\ldots,i_{p}}\  \text{is a dependent set of $M$} \Longrightarrow \set{\gamma_{i_{1}},\ldots,\gamma_{i_{p}}}\ \text{is linearly dependent}. \] 
The {\em circuit variety} of $M$ is defined as $V_{\mathcal{C}(M)}=\{\gamma:\text{$\gamma$ includes the dependencies of $M$}\}\subset \CC^{3d}$.
The ideal of $V_{\mathcal{C}(M)}$ is denoted by $I_{\CCC(M)}$.
\end{itemize}
\end{definition}

The main question studied in this work is the following:

\begin{question}\label{question generators}
Given a rank-three matroid $M$, find a finite generating set for the ideal $I_{M}$. 
\end{question}

As a first step toward resolving Question~\ref{question generators}, we seek methods for producing explicit polynomials in $I_{M}$. Since the inclusion $I_{\CCC(M)}\subseteq I_{M}$ is immediate, the difficulty lies in constructing additional generators beyond those arising from the circuit variety.
To this end, we rely on two techniques introduced in the literature: the Grassmann-Cayley algebra, as developed in \cite{Sidman}, and the geometric liftability method introduced in \cite{liwski2024pavingmatroidsdefiningequations}. These yield two subideals of $I_{M}$, denoted $G_{M}$ and $I_{M}^{\textup{lift}}$: 
\begin{itemize}
\item The ideal $G_{M}$ 
is generated by Grassmann-Cayley polynomials associated with triples of concurrent lines in $M$. See \S\ref{Grassmann Cayley ideal} for further details.
\item The ideal $I_{M}^{\textup{lift}}$ 
consists of polynomials obtained via the geometric liftability technique. See \S\ref{liftability}. 
\end{itemize}
More precisely, as described in~\cite{Sidman}, the Grassmann–Cayley algebra offers a framework for translating geometric conditions, such as the concurrence of lines in $\mathbb{P}^2$, into algebraic relations. Accordingly, each triple of concurrent lines in a matroid $M$ corresponds, in any realization of $M$, to a triple of concurrent lines in $\mathbb{P}^2$, yielding a polynomial contained in $I_M$.

Furthermore, given a rank-three matroid $M$ on $[d]$ and a collection of vectors $\{\gamma_{1},\ldots,\gamma_{d}\} \subset \mathbb{C}^3$ on a common hyperplane, one may ask under what conditions these vectors can be lifted, from a fixed vector $q \in \mathbb{C}^3$, to a full-rank collection in $V_{\CCC(M)}$. As shown in~\cite{liwski2024pavingmatroidsdefiningequations}, this liftability condition translates into algebraic relations, which yield polynomials in $I_M$.

\begin{example}
Consider the simple rank-three matroid 
in Figure~\ref{fig:combined} (Left). In any realization, the lines $12$, $34$, and $56$ are concurrent, and this geometric condition translates, via the Grassmann–Cayley algebra, into the vanishing of the following bracket polynomial 
$\corch{123}\corch{456} - \corch{124}\corch{356}$,
where $X$ is a $3\times 6$ matrix of indeterminants and $[ijk]$ denotes the determinant of the $3\times 3$ submatrix of $X$ whose columns are indexed by $i,j,k$. 
Thus, this polynomial belongs to the matroid ideal. See~\textup{\cite[Example~2.1.2]{Sidman}} for further details.

Now consider the \emph{quadrilateral set}, the simple rank-three matroid shown in Figure~\ref{fig:combined} (Right), with 
circuits of size three
$\{\{1,2,3\},\{1,5,6\},\{2,4,6\},\{3,4,5\}\}$. 
For any vector $q \in \mathbb{C}^3$, the $4$-minors of the following matrix lie in the matroid ideal (See~Notation~\ref{brackets} and 
\S\ref{liftability} for further details):
\begin{equation}\label{matrix quad}
\begin{pmatrix}
\corch{23q} & -\corch{13q} & \corch{12q} & 0 & 0 & 0 \\
\corch{56q} & 0 & 0 & 0 & -\corch{16q} & \corch{15q} \\ 
0 & \corch{46q} & 0 & -\corch{26q} & 0 & \corch{24q} \\
0 & 0 & \corch{45q} & -\corch{35q} & \corch{34q} & 0
\end{pmatrix}.
\end{equation}
\end{example}

Since complete generating sets are known for each of the ideals $I_{\CCC(M)}$, $G_{M}$, and $I_{M}^{\text{lift}}$, a natural approach to Question~\ref{question generators} is to identify families of matroids for which the sum of these ideals equals~$I_{M}$.

\begin{question}\label{question identify}
Let $M$ be a simple rank-three matroid. Let $G_{M}$ denote the ideal generated by the Grassmann-Cayley polynomials associated with all triples of concurrent lines in $M$, and let $I_{M}^{\textup{lift}}$ denote the ideal consisting of polynomials obtained via the geometric liftability technique. Identify all matroids $M$ for which the following equality holds:
\begin{equation}\label{equality of ideals}
I_{M} = \sqrt{I_{\CCC(M)} + G_{M} + I_{M}^{\textup{lift}}}.
\end{equation}
\end{question}

Our main results exhibit families of matroids $M$ for which Question~\ref{question identify} has an affirmative answer. In this work, all references to generating sets of ideals are taken up to radical. For simplicity, we omit the phrase “up to radical”. 
We now outline the general strategy used to establish this equality.

\begin{strategy}\label{strategy}
To prove Equation~\eqref{equality of ideals}, we proceed as follows:
\begin{enumerate}
\item We establish the equivalent equality of varieties:
\begin{equation}\label{union of varieties}
V_{M}=V_{\CCC(M)}\cap V(G_M)\cap V(I_{M}^{\textup{lift}}).
\end{equation}

\item The inclusion~$\subseteq$ in~\eqref{union of varieties} is clear. To prove the reverse inclusion, let $\gamma\in V_{\CCC(M)}\cap V(G_M)\cap V(I_{M}^{\textup{lift}})$. We distinguish cases based on which component of $V_{\CCC(M)}$ contains $\gamma$.

\item In each case, we construct an arbitrarily small perturbation $\widetilde{\gamma}\in \Gamma_{M}$, showing that
\[\gamma \in \closure{\Gamma_{M}}^{\textup{Euclidean}}=\closure{\Gamma_{M}}^{\textup{Zariski}}=V_{M}.\]
This uses the fact that the Euclidean and Zariski closures of the realization space coincide.
\end{enumerate}
\end{strategy}

In this work, we concentrate on addressing Questions~\ref{question generators} and~\ref{question identify} for three specific families of matroids of rank three: Cactus configurations, Pascal configuration, and Pappus configuration; see Definition~\ref{cact irr}, Figure~\ref{fig:pascal} (b), and Figure~\ref{fig:Pappus, Pappus, I_i, J_i} (Left).
By applying Strategy~\ref{strategy}, along with several additional techniques tailored to each case, we obtain the following main results:

\begin{theoremA}\label{thm:A}
The equality in~\eqref{equality of ideals} holds for the following matroids:
\begin{itemize}
\item Cactus configurations\hfill\textup{(Theorem~\ref{thm: main theorem cactus matroid ideal})}
\item Pascal configuration\hfill\textup{(Theorem~\ref{generators pascal})}
\item Pappus configuration\hfill\textup{(Theorem~\ref{thm:Pappus})}
\end{itemize}
\end{theoremA}
Based on Theorem~\ref{thm:A}, one can derive explicit generating sets, up to radical, for the ideals corresponding to the Pascal and Pappus configurations. The table below summarizes the number of generators of each type: circuit polynomials, Grassmann–Cayley polynomials, and lifting polynomials. For detailed expressions and geometric interpretations of these generators, see Remarks~\ref{remark pascal} and~\ref{remark pappus}.

\begin{table}[h!]
\centering
\begin{tabular}{|l|c|c|c|}
\hline
{Configuration} & {Circuit polynomials} & {Grassmann–Cayley polynomials} & {Lifting polynomials} \\
\hline
Pascal & 7 & 7 & 708{,}588 \\ 
Pappus & 9 & 9 & 2{,}361{,}960 \\
\hline
\end{tabular}
\end{table}
The table highlights that the number of lifting polynomials in our generating sets is remarkably large, an observation that generally holds across configurations. This naturally leads to the question of whether smaller generating sets can be constructed. Although we do not pursue this direction in the present work, we pose the following as a guiding question for future research:

\begin{question}
Develop efficient methods for constructing a minimal generating set for the ideal $I_{M}^{\textup{lift}}$. 
\end{question}

Several previous works have provided complete sets of defining equations for specific families of matroids. 
In \cite{pfister2019primary, clarke2024liftablepointlineconfigurationsdefining, liwski2024pavingmatroidsdefiningequations}, it was shown that for certain families of configurations, the ideals $I_{\CCC(M)}$ and $I_{M}^{\text{lift}}$ suffice to generate the ideal $I_{M}$, up to radical.
In \cite{liwski2024pavingmatroidsdefiningequations, clarke2021matroid}, other families were identified for which the ideals $I_{\CCC(M)}$ and $G_M$ generate $I_{M}$.
To our knowledge, this is the first instance in which matroids are identified for which all three ideals $I_{\CCC(M)},I_{M}^{\text{lift}}$ and $G_M$ are needed to generate $I_M$.

In addition, we establish structural decomposition results for the matroid and circuit varieties associated with these configurations, namely Cactus, Pascal and Pappus configurations.
\begin{theoremB}\label{thm:B}
The following statements hold for cactus and Pascal configurations:

\begin{itemize}
\item Let $M$ be any cactus configuration. Then $M$ is realizable, and its matroid variety $V_{M}$ is irreducible. Moreover, if $Q_{M}$ denotes the set of points in $M$ that lie on at least three lines, the circuit variety $\VCM$ has at most $2^{\size{Q_{M}}}$ irredundant irreducible components, each obtained by setting a subset of points in $Q_{M}$ to be loops. \hfill (Theorems~\ref{cact irr} and~\ref{decomposition cactus})

\item The circuit variety of the Pascal configuration $N$ has the following irreducible decomposition:
\[
V_{\CCC(N)} = V_{N} \cup V_{U_{2,9}} \cup \bigcup_{i=7}^{9} V_{N(i)}, \tag{Theorem~\ref{proposition: decomposition pascal configuration}}
\]
where $U_{2,9}$ is the uniform matroid 
and $N(i)$ is obtained from $N$ by setting $i$ to be a loop.
\end{itemize}
\end{theoremB}

\begin{example}
Consider the cactus configuration shown in Figure~\ref{fig:preliminaries gemengd} (Right), which determines a rank-three matroid $M$ on the ground set $[14]$. By Theorem~\ref{thm:B}, the matroid $M$ is realizable and the variety $V_M$ is irreducible. Moreover, since $Q_M = \{1,2,3\}$, the variety $V_{\CCC(M)}$ has at most eight irreducible components, each obtained by setting a subset of points in $\{1,2,3\}$ to be loops.
\end{example}

While the results presented herein may seem limited to particular families of matroids, it is important to note, as observed in \cite{pfister2019primary}, that determining a generating set for $I_{M}$ is notoriously difficult, with few general methods available for constructing such polynomials. The authors describe questions concerning the matroid ideal as having {\em paved the road to Hell}, descending into an {\em abyss}. 
The matroid ideals are only known for special classes of matroids. We hope that the ideas and techniques developed in this work will advance the understanding of matroid varieties and help identify the conditions under which Questions~\ref{question generators} and~\ref{question identify} have positive answers.


\smallskip

Before reviewing the contents of the paper we would like to comment on some
related works. 
In \cite{pfister2019primary}, the authors developed an algorithm aimed at computing defining equations for the matroid variety associated with the $3\times 4$ grid configuration, which contains 16 circuits of size three. Using {\tt Singular}, they provided a finite generating set for the corresponding ideal. However, the computations stretch the capabilities of current computer algebra systems, and the resulting components lack a clear combinatorial interpretation.
In \cite{clarke2024liftablepointlineconfigurationsdefining}, the authors obtained complete sets of defining equations for the matroid varieties of the quadrilateral set and the $3\times 4$ grid, giving a geometric interpretation for such polynomials.
In \cite{liwski2024pavingmatroidsdefiningequations}, the geometric liftability technique was used to define the ideal $I_{M}^{\textup{lift}}\subset I_{M}$. It was shown that, for paving matroids without points of degree greater than two, the ideal $I_{\CCC(M)}+I_{M}^{\textup{lift}}$ defines the associated matroid variety.
In \cite{Sidman}, the Grassmann–Cayley algebra was employed to define the ideal $G_{M}\subset I_{M}$. Using this approach, the authors constructed seven polynomials in the matroid ideal of Pascal configuration.
In~\cite{liwski2024pavingmatroidsdefiningequations}, it was shown that the ideal $I_{\CCC(M)}+G_M$ is equal to the associated matroid ideal of any forest configuration $M$, up to radical.
For decomposing the circuit variety $V_{\CCC(M)}$, we adopt the algorithm introduced in \cite{liwski2025minimal,liwski2025efficient}.

\smallskip
Although matroid varieties are widely studied in algebraic geometry for their intricate geometric structure, they also arise naturally in a range of other settings, including determinantal varieties~\cite{bruns2003determinantal, clarke2021matroid, herzog2010binomial, pfister2019primary, ene2013determinantal}, rigidity theory~\cite{jackson2024maximal, whiteley1996some, graver1993combinatorial, maximum}, and conditional independence models~\cite{Studeny05:Probabilistic_CI_structures, DrtonSturmfelsSullivant09:Algebraic_Statistics, hocsten2004ideals, clarke2022conditional, caines2022lattice}. In these contexts, one often encounters the problem of decomposing certain ideals such as conditional independence ideals or determinantal ideals into primary components. Notably, matroid ideals frequently appear as such primary components in such decompositions. As a result, understanding the defining equations of matroid varieties, or equivalently their associated matroid ideals, becomes essential for determining the primary decomposition of these broader classes of ideals.

\smallskip
We conclude the introduction with an outline of the paper. Section~\ref{preliminaries} reviews the necessary background on matroids and their associated varieties. In Section~\ref{section cactus}, we introduce and study cactus configurations. We prove that every such configuration corresponds to a nilpotent realizable matroid, and that its associated matroid variety is irreducible. Moreover, we determine a complete generating set, up to radical, for the matroid ideals of cactus configurations. Section~\ref{section pascal and papus} focuses on the Pascal and Pappus configurations, providing generating sets, up to radical, for their respective matroid ideals.

\section{Preliminaries}\label{preliminaries}

This section provides a concise overview of properties of matroids and their associated varieties. For more details, we refer the reader to \cite{Oxley, gelfand1987combinatorial, piff1970vector}. 
Throughout, we adopt the notation $[n] = \{1, \ldots, n\}$ and $\textstyle \binom{[d]}{n}$ to denote the set of $n$-element subsets of $[d]$.

\subsection{Matroids}\label{matroids}
A matroid $M$ consists of a ground set $[d]$ together with a collection $\mathcal{I}$ of subsets of $[d]$, called independent sets, that satisfy the following three axioms: the empty set is independent, i.e., $\emptyset \in \mathcal{I}$; the hereditary property, meaning if $I \in \mathcal{I}$ and $I' \subset I$, then $I' \in \mathcal{I}$; and the exchange property, which states that if $I_1, I_2 \in \mathcal{I}$ with $|I_1| < |I_2|$, then there exists an element $e \in I_2 \setminus I_1$ such that $I_1 \cup \{e\} \in \mathcal{I}$.

The rank of $M$ is the size of its largest independent subset. Since we focus exclusively on matroids of rank at most three, all definitions will be stated in that context. 
Let $M$ be a matroid of rank three on the ground set $[d]$. We introduce the following notions:
\begin{itemize}
\item A subset of $[d]$ is called {\em dependent} if it is not independent. The collection of all dependent sets of $M$ is denoted by $\mathcal{D}(M)$.
\item A subset of $[d]$ is called a {\em circuit} if it is a minimally dependent subset. The set of all circuits is denoted by $\mathcal{C}(M)$.
\item A {\em basis} is an independent subset of $[d]$ of size three. The set of all bases is denoted by $\mathcal{B}(M)$.

\item An element $x\in [d]$ is called a {\em loop} if $\{x\}\in \mathcal{C}(M)$.
\item Two elements $x,y\in [d]$ are said to be {\em parallel}, or form a {\em double point}, if $\{x,y\}\in \mathcal{C}(M)$.
\item $M$ is called {\em simple} if it contains no loops or parallel elements.
\item The {\em restriction of $M$ to $S$} is the matroid 
 on $S$ whose independent sets are the independent sets of $M$ contained in $S$.
 The {\em deletion} of $S$, denoted $M\backslash S$, is defined as $M|([d]\backslash S)$.
\item A {\em line} is defined as a maximal dependent subset of $[d]$, in which every subset of three elements is dependent.  We denote the set of all lines of $M$ by $\mathcal{L}_M$, or simply $\mathcal{L}$ when the context is clear.  

\item Elements in $[d]$ are called {\em points}. For any point $p \in [d]$, let $\mathcal{L}_{p}$ denote the set of lines containing $p$. The {\em degree} of $p$ is defined as $\size{\mathcal{L}_{p}}$.
\end{itemize}

\begin{example}\normalfont\label{uniform 3}
The uniform matroid $U_{2, d}$ on the ground set $[d]$ of rank $2$ is defined as follows: each subset $S\subset [d]$ with $|S| \leq 2$ is independent, while those with $|S| > 2$ are dependent. 
\end{example}


\begin{example}\label{three lines}
Consider the configuration depicted in Figure~\ref{fig:combined} (Left). It defines a simple rank-three matroid on the ground set \([7]\), whose collection of lines is $\mathcal{L} = \{ \{1,2,7\}, \{3,4,7\}, \{5,6,7\} \}$.
The configuration in Figure~\ref{fig:combined} (Right) gives rise to a simple rank-three matroid on \([6]\), with set of lines 
$ \{ \{1,2,3\}, \{1,5,6\}, \{2,4,6\}, \{3,4,5\}\}$.
We refer to this matroid as the \emph{quadrilateral set}, denoted by \(\text{QS}\).
\end{example}

\begin{figure}[h]
    \centering
    \includegraphics[width=0.4\linewidth]{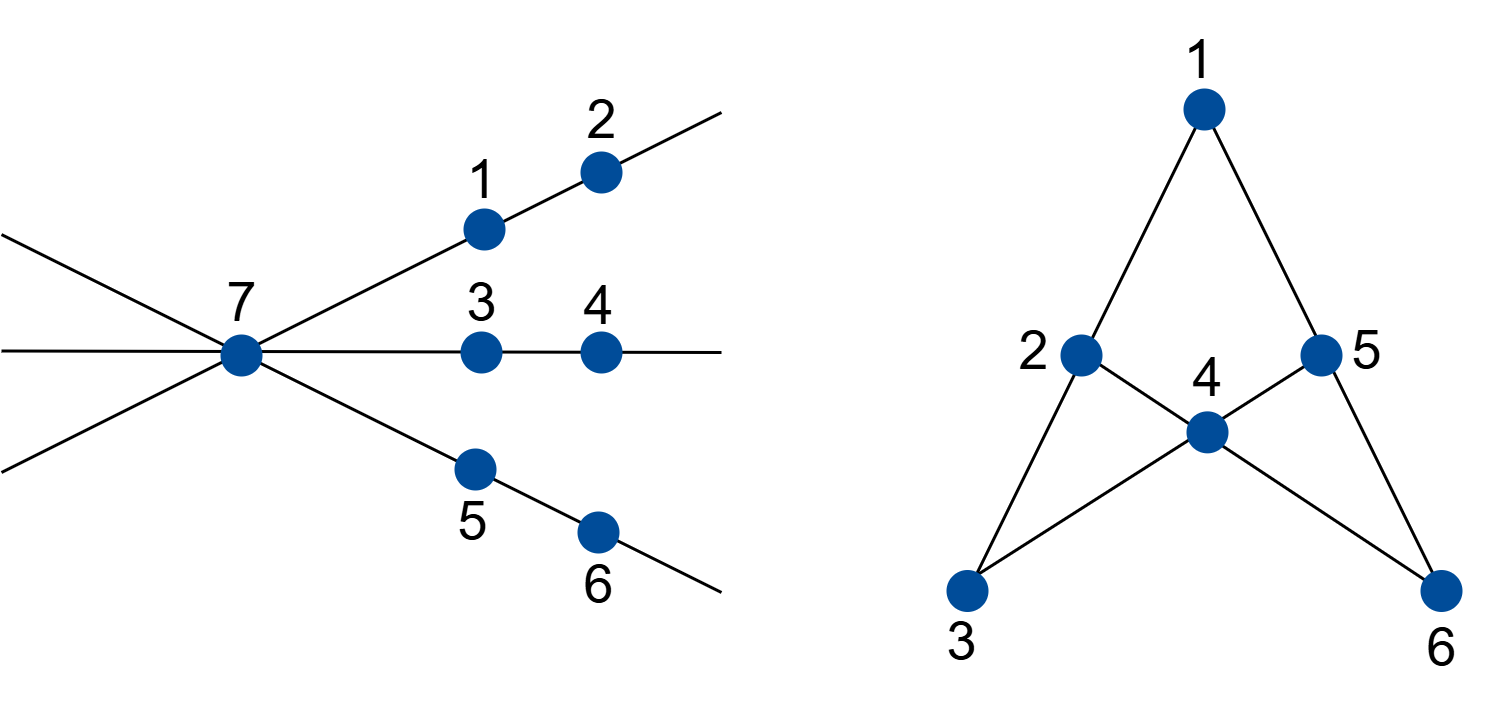}
    \caption{(Left) Three concurrent lines; (Right) Quadrilateral set.}
    \label{fig:combined}
\end{figure}

As in the previous example, we will consistently identify each configuration 
with its associated rank-three matroid. We define two families of matroids that play a key role in this note; see \cite{liwski2024pavingmatroidsdefiningequations}.

\begin{definition}\normalfont \label{nil sol}
For a 
simple rank-three matroid $M$ on $[d]$, we define 
\[S_{M}=\{p\in [d]: \size{\mathcal{L}_{p}}\geq 2\},\quad \text{and} \quad Q_{M}=\{p\in [d]: \size{\mathcal{L}_{p}}\geq 3\}.\] 
We then consider the following chains of submatroids of $M$: 
\begin{equation*}
\begin{aligned}
&M_{0}=M, \ \quad M_{1}=M|S_{M},\quad \text{and }\quad M_{j+1}=M|S_{M_{j}} \quad \ \text{ for all $j\geq 1$.}\\
& M^{0}=M,\quad M^{1}=M|Q_{M},\quad \text{and }\quad M^{j+1}=M|Q_{M^{j}} \quad \text{ for all $j\geq 1$.}
\end{aligned}
\end{equation*}
We say that $M$ is {\em nilpotent} if $M_{j}=\emptyset$ for some $j$, and 
{\em solvable} if $M^{j}=\emptyset$ for some $j$.
\end{definition}

\subsection{Matroid and circuit varieties}

In this subsection, we describe several properties of the matroid and circuit varieties of a matroid $M$ of rank at most three, denoted $V_{M}$ and $V_{\CCC(M)}$, respectively. These varieties are defined in Definition~\ref{first definition}. 

\begin{notation}\label{brackets}
We fix the following  notation that will be used throughout the paper.
\begin{enumerate}
\item For a $3\times d$ matrix $X=\pare{x_{i,j}}$ of indeterminates, we denote by $\CC[X]$ the polynomial ring in the variables $x_{ij}$. Given subsets $A\subset[3]$ and $B\subset[d]$ with $\size{A}=\size{B}$, we denote by $[A|B]_{X}$ the minor of $X$ determined by the rows indexed by $A$ and columns indexed by $B$. 
\item \label{determinant} For vectors $v_{1},v_{2},v_{3}\subset \CC^{3}$, we write $[v_{1},v_{2},v_{3}]$ for the determinant of the matrix whose columns are the vectors $v_1,v_{2},v_3$.
\item \label{brackets notation}To any matroid $M$ of rank at most three on $[d]$, we associate $X$, a $3\times d$ matrix of indeterminates. For any subset $\{i_{1},\ldots,i_{k}\}\subset [d]$ with $k\leq 3$, and vectors $v_{1},\ldots,v_{3-k}\in \CC^{3}$, 
we define \[[i_{1},\ldots,i_{k},v_{1},\ldots,v_{3-k}]\in \CC[X],\] as the determinant of the $3\times 3$ submatrix of $X$ whose columns are indexed by $i_{1},\ldots,i_{k}$ 
and $v_{1},\ldots,v_{3-k}$. Note that $[i_{1},\ldots,i_{k},v_{1},\ldots,v_{3-k}]$ is a polynomial in $\CC[X]$, rather than a number.
\end{enumerate}
\end{notation}

The following introduces the circuit ideal $I_{\CCC(M)}$ and presents explicit equations that generate it.

\begin{definition}\normalfont\label{cir}
Let $M$ be a matroid of rank at most three on $[d]$.
Consider the $3\times d$ matrix $X=\pare{x_{i,j}}$ of indeterminates. The {\em circuit ideal}, for which $V_{\CCC(M)}=V(I_{\CCC(M)})$, is defined as 
$$ I_{\CCC(M)} = \{ [A|B]_X:\ B \in \CCC(M),\ A \subset [3],\ \text{and}\ |A| = |B|\}.$$
\end{definition}

\begin{remark}\label{generating set ICM}
Since $V_{M}\subset V_{\CCC(M)}$, we have $I_{\CCC(M)}\subset I_M$. Further, if $M$ is a simple rank-three matroid on $[d]$, then the ideal $I_{\CCC(M)}$ is generated by the following bracket polynomials defined in  Notation~\ref{brackets}(iii):
\[\{[i,j,k]:\{i,j,k\}\in \CCC(M)\}.\]
\end{remark}

We recall the following result from \cite{liwski2024pavingmatroidsdefiningequations}, 
which will be used in the subsequent sections.  
\begin{theorem}\label{nil coincide}
Let $M$ be a simple rank-three matroid on $[d]$. 
\begin{itemize}
\item[{\rm (i)}] If $M$ is nilpotent and has no points of degree greater than two, then $V_{\mathcal{C}(M)}=V_{M}$.
\item[{\rm (ii)}] If $M$ is solvable, then $V_{M}$ is irreducible.
\end{itemize}
\end{theorem}

Since the Euclidean and Zariski closures of the realization space $\Gamma_{M}$ coincide, our general strategy to show that a collection of vectors $\gamma$ lies in $V_{M}$ is to demonstrate that an arbitrarily small perturbation of points can be applied to $\gamma$ to obtain an element in $\Gamma_{M}$. We formalize this in the following definition from \cite{hocsten2004ideals}, where a \emph{perturbation} refers to a motion that can be made arbitrarily small. This concept will be used repeatedly throughout the paper.

\begin{definition}
Let $\gamma \subset \mathbb{C}^3$ be a finite collection of vectors, and let $X$ denote a specific property. We say that a \emph{perturbation} can be applied to $\gamma$ to obtain a new collection of vectors satisfying property $X$ if, for every $\epsilon > 0$, it is possible to choose, for each vector $v \in \gamma$, a vector $\widetilde{v}$ such that $\lVert v - \widetilde{v} \rVert < \epsilon$, in such a way that the new collection satisfies property $X$.

When discussing a perturbation of a $k$-dimensional subspace $S \subset \mathbb{C}^n$, we consider $\{v_{1},\ldots,v_{k}\}$ a basis of it, and apply a perturbation to $\{v_i : i \in [k]\}$ to obtain another $k$-dimensional subspace.
\end{definition}

\subsection{Liftability}\label{liftability}

We now define {\em liftable} matroids. Throughout, let $M$ be a simple rank-three matroid on $[d]$.

\begin{definition}\label{def liftable} We introduce the following notions:
\begin{itemize}
\item Consider a collection of vectors $\gamma = \{\gamma_{i} : i\in [d]\}\subset \CC^{3}$ and a vector $q\in \CC^{3}$. We say that  $\widetilde{\gamma}=\{\widetilde{\gamma}_{i}:i\in [d]\}\subset \CC^{3}$ is a {\em lifting} of $\gamma$ from the vector $q$ if, for each $i\in [d]$, there exists $z_{i}\in \CC$ such that $\widetilde{\gamma}_{i}=\gamma_{i}+z_{i}q$. Furthermore, we say that the lifting is {\em non-degenerate} if not all vectors of $\widetilde{\gamma}$ lie within the same hyperplane.
\item The collection $\gamma$ is said to be liftable from the vector $q$, if the vectors $\gamma_i$ can be lifted from the point $q$ to a collection $\widetilde{\gamma} \in \VCM$ with rank($\widetilde{\gamma}) = 3$.
\item A simple rank-three matroid $M$ is \textit{liftable} if, for any rank-two collection of vectors $\gamma = \{\gamma_i : i \in [d]\}$ in $\mathbb{C}^3$, there exists a non-degenerate lifting of $\gamma$ to a collection of vectors $\widetilde{\gamma} = \{\gamma_i+z_{i}q : i\in [d]\}\in V_{\mathcal{C}(M)}$, from any vector $q$ which is in general position with respect to $\gamma$.
\end{itemize}
\end{definition}

In the next construction, we define the {\it liftability matrix}.

\begin{definition}\label{matrix lift}
Let $M$ be a simple rank-three matroid on $[d]$, and let $q\in \mathbb{C}^{3}$. We define the \textit{liftability matrix} $\mathcal{M}_{q}(M)$ of $M$ and $q$, as the matrix with columns indexed by $[d]$ and rows indexed by the circuits of size three of $M$. The entries of this matrix are defined as follows: for each circuit $c=\{c_{1},c_{2},c_{3}\}\subset [d]$, the $c_{1}^{\text{th}},c_{2}^{\text{th}}$ and $c_{3}^{\text{th}}$ coordinate of the corresponding row are given by the polynomials 
\[ 
[c_{2},c_{3},q], \quad -[c_{1},c_{3},q],\quad \text{and} \quad [c_{1},c_{2},q],
\]
where for the brackets we are using  Notation~\ref{brackets}\ref{brackets notation}. 
The other entries of the row are set to $0$. Note that the entries of the matrix are polynomials, not numbers. Given a collection of vectors $\gamma=\{\gamma_{1},\ldots,\gamma_{d}\}\subset \CC^{3}$, we denote by $\liftmat$ the matrix obtained from $\mathcal{M}_{q}(M)$ by substituting, for each $i\in [d]$, the three variables associated to index $i$ with the entries of $\gamma_{i}$.
\end{definition}

\begin{definition}
    Consider a simple rank-three matroid $M$. We define the {\it lifting ideal}, denoted $I_M^{\text{lift}}$, as the ideal generated by all the $(|N|-2)$-minors of the liftability matrices $\mathcal{M}_q(N)$, where $N$ ranges over all full-rank submatroids of $M$ and $q$ varies over $\CC^3$.
\end{definition}

We recall some of the main properties of the lifting ideal; see Lemma~3.6 and Theorem~3.9 in \cite{liwski2024pavingmatroidsdefiningequations}.

\begin{theorem} \label{thm: gamma i V_M^lift => liftable}
Let $M$ be a simple rank-three matroid on $[d]$.
Then $I_M^\textup{lift} \subseteq I_M$.
Additionally, consider $\gamma=\{\gamma_{1},\ldots,\gamma_{d}\}\subset \CC^{3}$, a collection of vectors in $V(I_{M}^{\text{lift}})$. Then $\gamma$ is liftable from any $q\in \CC^{3}$, and such a lifting can be made arbitrarily small. The same holds for any full-rank submatroid of $M$.
\end{theorem}

\subsection{Grassmann-Cayley algebra}\normalfont

We now review the notion of Grassmann-Cayley algebra from \textup{\cite{AlgorithmsInInvariantTheory}}. The Grassmann-Cayley algebra is the exterior algebra $\bigwedge (\mathbb{C}^{d})$, equipped with two operations: the \textit{join} and the \textit{meet}. The join, or extensor, of vectors $v_{1}, \ldots, v_{k}$ is denoted by $v_{1}  \cdots v_{k}$. The meet operation, denoted by $\wedge$, is defined for two extensors $v = v_{1} \cdots v_{k}$ and $w = w_{1} \cdots w_{j}$, with lengths $k$ and $j$ respectively, where $j + k \geq d$, as:
$$v\wedge w=\sum_{\sigma\in \mathcal{S}(k,j,d)}
\corch{v_{\sigma(1)}\ldots v_{\sigma(d-j)}w_{1}\ldots w_{j}}\cdot v_{\sigma(d-j+1)}\cdots v_{\sigma(k)}$$
where $ \mathcal{S}(k,j,d)$ denotes the set of all permutations of $\corch{k}$ that satisfy $\sigma(1)<\cdots <\sigma(d-j)$ and $\sigma(d-j+1)<\cdots <\sigma(d)$ and the bracket denotes the determinant, see Notation~\ref{brackets}\ref{determinant}. When $j + k < d$, the meet is defined as $0$. There is a correspondence between the extensor $v = v_{1} \cdots v_{k}$ and the subspace $\overline{v} = \langle v_{1}, \ldots, v_{k} \rangle$ that it generates. This correspondence satisfies the following properties:

\begin{lemma}\label{klj}

Let $v=v_{1}\cdots v_{k}$ and $w=w_{1}\cdots w_{j}$ be two extensors with $j+k\geq d$. Then we have:
\begin{itemize}
\item The extensor $v$ is equal to zero if and only if the vectors $v_{1},\ldots,v_{k}$ are linear dependent.
\item Any extensor $v$ is uniquely determined by $\overline{v}$ up to a scalar multiple.
\item The meet of two extensors is again an extensor.
\item We have that $v\wedge w\neq 0$ if and only if $\ip{\overline{v}}{\overline{w}}=\CC^{d}$. In this case, we have $\overline{v}\cap\overline{w}=\overline{v\wedge w}.$
\end{itemize}
\end{lemma}

\begin{example}\label{gc3}
Consider the simple rank-three matroid depicted on Figure~\ref{fig:pascal} (b). Since, in any realization, the lines $\{26, 35, 78\}$ are concurrent, the point of intersection of $26$ and $35$ lies on $78$. Thus, the concurrency of the lines $\{26, 35, 78\}$ is equivalent to the vanishing of the following polynomial:
$$(26)\wedge (35)\vee 78=(\corch{235}6-\corch{635}2)\vee 78=\corch{235}\corch{678}-\corch{635}\corch{278}.$$
Therefore, we know that this polynomial is contained in the ideal of the matroid variety.
\end{example}

\subsubsection{Grassmann-Cayley ideal}\label{Grassmann Cayley ideal} 

We define the Grassmann-Cayley ideal using a construction outlined in  \cite{liwski2024pavingmatroidsdefiningequations}.
Let $M$ be a simple rank-three matroid on $[d]$ and let $x\in [d]$. Consider distinct lines $l_{1}\neq l_{2}$ containing $x$ and points $p_{1},p_{2}\in l_{1}$, $p_{3},p_{4}\in l_{2}$. Let $P\in I_{M}$.
Since $x=l_{1}\cap l_{2}$, for any realization $\gamma \in \Gamma_{M}$, we have:
\[
\gamma_{x}=\gamma_{p_{1}}\gamma_{p_{2}}\wedge \gamma_{p_{3}}\gamma_{p_{4}}=\corch{\gamma_{p_{1}},\gamma_{p_{2}},\gamma_{p_{3}}}\gamma_{p_{4}}- \corch{\gamma_{p_{1}},\gamma_{p_{2}},\gamma_{p_{4}}}\gamma_{p_{3}}.  \]
Consequently, the polynomial  obtained from $P$ by replacing the variable $x$ 
with:
\[\corch{p_{1},p_{2},p_{3}}p_{4}-\corch{p_{1},p_{2},p_{4}}p_{3}\]
is also in $I_{M}$.

\begin{definition}
\label{gm} 
Given a simple rank-three matroid $M$, the ideal $G_{M}$ is constructed as follows:
\begin{itemize}
\item Define the set $X_{0}$ as the polynomials generating the ideal $I_{\mathcal{C}(M)}$, i.e., the brackets corresponding to the circuits of size three of $M$, see Remark~\ref{generating set ICM}. For $j \geq 1$, recursively define the set $X_{j}$ as the polynomials obtained from those in $X_{j-1}$ by modifying some of their variables, potentially leaving them unchanged, according to the procedure described above. Note that $X_{0} \subset X_{1} \subset \ldots$.
\item For $j \geq 0$, define the ideal $I_{j}$ as the ideal generated by the polynomials in $X_{j}$. Note that $I_{0} \subset I_{1} \subset \ldots$. Since the polynomial ring is Noetherian, this chain of ideals stabilizes. We denote the stabilized ideal by $G_{M}$.
\end{itemize}
\end{definition}

We recall the following result from \cite{liwski2024pavingmatroidsdefiningequations}, which will be used repeatedly in the proofs that follow.

\begin{proposition} \label{inc G}
Let $M$ be a simple rank-three matroid on $[d]$. 
Then $G_{M}\subset I_{M}$.
Moreover, let $l_{1},l_{2},l_{3}$ be lines of $M$ containing a common point $x\in [d]$. Then 
$\gamma_{l_{1}}\wedge \gamma_{l_{2}}\wedge\gamma_{l_{3}}=0$ 
for any $\gamma\in V_{\CCC(M)}\cap V(G_M)$, meaning that the lines $\gamma_{l_{1}},\gamma_{l_{2}},\gamma_{l_{3}}$ are concurrent in $\mathbb{P}^{2}$.
(If $\gamma_{x}=0$, one may redefine $\gamma_{x}$ to be the point in $\mathbb{P}^{2}$ contained in the intersection of the three lines $\gamma_{l_{1}},\gamma_{l_{2}}$ and $\gamma_{l_{3}}$).
\end{proposition}

\section{Cactus configurations}\label{section cactus}

In this section, we introduce cactus configurations and find a generating set for their associated matroid ideals, as well as the decomposition of their circuit varieties.

\smallskip
We begin by fixing the notation that will be used throughout this section. Throughout, we alternate freely between projective and affine language.

\begin{notation} 
Let $M$ be a simple matroid of rank three on $[d]$. 
\begin{itemize}
\item For a collection of vectors $\gamma=\{\gamma_{1},\ldots,\gamma_{d}\}\in V_{\CCC(M)}$ and any line $l$ of $M$ we denote by $\gamma_{l}\subset \CC^{3}$ the following two-dimensional subspace
$\gamma_{l}=\sspan \{\gamma_{i}:i\in l\}$.
\item The subspace $\gamma_{l}$ corresponds to a unique line in the projective plane $\mathbb{P}^{2}$, and we will therefore often refer to $\gamma_{l}$ simply as a {\em line}.
\item Given that each nonzero vector $\gamma_{i}$ determines a point in $\mathbb{P}^{2}$, we use the terms {\em vector} and {\em point} interchangeably when $\gamma_{i}\neq 0$, depending on whether we are viewing it on $\CC^{3}$ or $\mathbb{P}^{2}$.
\item When two nonzero vectors $\gamma_{i}$ and $\gamma_{j}$ are linearly dependent, we may say they are the {\em same point} or that they {\em coincide}, denoted with $=$, reflecting that they are the same point in $\mathbb{P}^{2}$.
\item If $\gamma_{i}=0$, we refer to $\gamma_{i}$ as a {\em loop}.
\item Although it may seem natural to work entirely in the projective setting, the possible occurrence of zero vectors does not allow this simplification.
\end{itemize}
\end{notation}

\subsection{Irreducibility of matroid varieties of cactus configurations}
In this subsection, we introduce the concept of cactus configurations and establish the irreducibility of their associated matroid varieties. We begin by introducing some necessary notions.

\begin{definition}\label{definition n-cycle}
Let $M$ be a simple matroid of rank at most three on $[d]$.
\begin{itemize}
\item  If $M$ is the uniform matroid $U_{2,d}$, we refer to $M$ as a {\em line}; see Definition~\ref{uniform 3}.
\item Suppose $M$ has $n$ lines. We say that $M$ is a {\em cycle} if there exists a subset of points $\{p_{1},\ldots,p_{n}\}\subset [d]$ together with an ordering $\{l_{1},\ldots,l_{n}\}$ of its lines, such that:
\begin{enumerate}
\item Each point $p_{i}$	
  is incident to precisely two lines, satisfying $\mathcal{L}_{p_{i}}=\{l_{i},l_{i+1}\}$, for each $i\in [n]$, where we identify $l_{n+1}$ with $l_{1}$.
\item Every other point $p\in [d]\backslash \{p_{1},\ldots,p_{n}\}$ is incident to exactly one line, meaning $\size{\mathcal{L}_{p}}=1$.
\end{enumerate}
\end{itemize}
\end{definition}

\begin{definition}
Let $M$ and $N$ be simple matroids of rank at most three on $[d_{1}]$ and $[d_{2}]$, respectively, and consider points $p\in [d_{1}]$ and $q\in [d_{2}]$. We define the {\em free gluing} of $M$ and $N$ at $p$ and $q$ as the simple rank-three matroid $M \amalg_{p,q} N$ on the ground set 
\[([d_{1}]\backslash \{p\})\amalg ([d_{2}]\backslash \{q\})\cup \{P\},\]
where $P$ is a new point that identifies both $p$ and $q$. The set of lines of $M \amalg_{p,q} N$ is given by:
\[
\mathcal{L}_{M \amalg_{p,q} N}= \{(l\backslash \{p\})\cup \{P\}:l\in \mathcal{L}_{p}\}\cup
\{(l\backslash \{q\})\cup \{P\}:l\in \mathcal{L}_{q}\}
\cup (\mathcal{L}_{M}\backslash \mathcal{L}_{p})\cup (\mathcal{L}_{N}\backslash \mathcal{L}_{q}).
\]
\end{definition}

We now illustrate the above definition with  two examples as follow. 

\begin{example}
Let $M$ denote the $3 \times 3$ grid configuration, and let $F$ be the Fano plane, as depicted in Figure~\ref{fig:preliminaries gemengd}. Figure~\ref{example free gluing} (Left) shows the free gluing $M\amalg_{7,7'}F$ at the points $7$ and $7'$.
    
Similarly, consider the configuration consisting of three concurrent lines and the quadrilateral set shown in Figure~\ref{fig:combined}. Their free gluing along the points $3$ and $3'$ is shown in Figure~\ref{example free gluing} (Right).
\end{example}

\begin{figure}[h]
    \centering
    \includegraphics[width=0.6\linewidth]{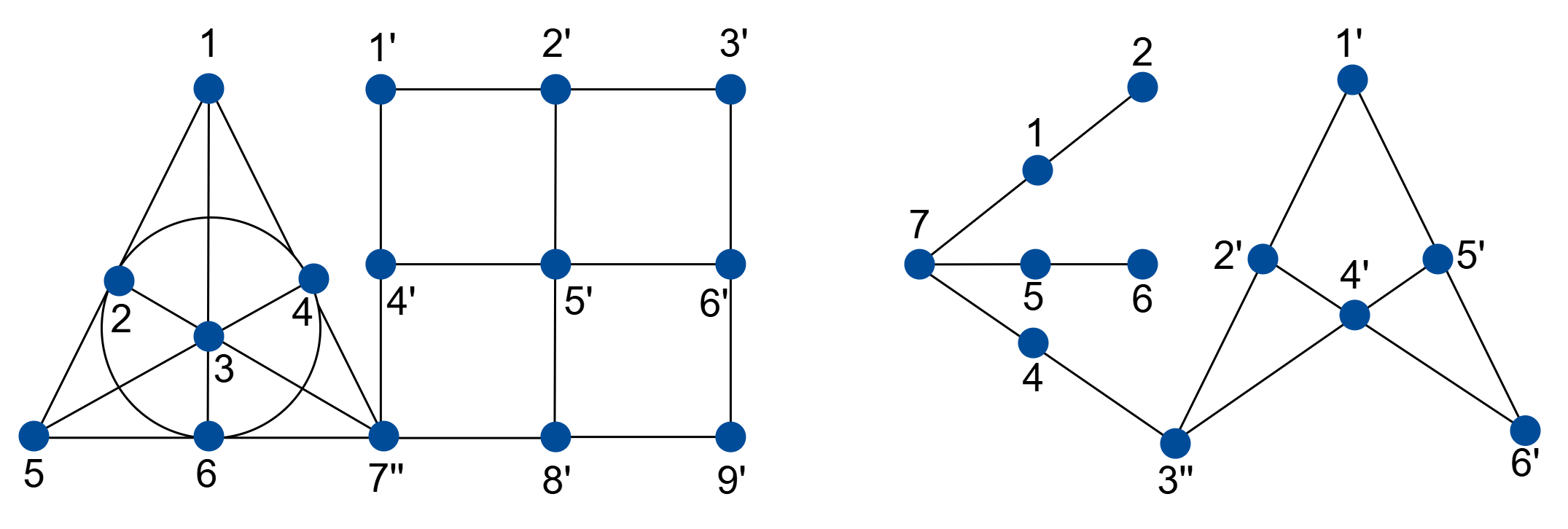}
    \caption{(Left) Free gluing of the Fano plane and the $3 \times 3$ grid; (Right) Free gluing of three concurrent lines and the quadrilateral set.}
    \label{example free gluing}
\end{figure}

We can think of the free gluing $M \amalg_{p,q} N$ as the simple rank-three matroid obtained by gluing $M$ and $N$ at the points $p$ and $q$ in the most ``free" or independent way, by preserving all dependencies of $M$ and $N$ without introducing any new ones.
This notion of free gluing is closely related to the classical concept of {\em amalgamation} of matroids~\cite{poljak1984amalgamation}. Given matroids $M$ and $N$ on ground sets $[d_{1}]$ and $[d_{2}]$, an amalgamation is a matroid whose restrictions to $[d_{1}]$ and $[d_{2}]$ agree with $M$ and $N$, respectively.

\medskip
We are now prepared to define cactus configurations.

\begin{definition}\label{cactus}
A simple rank-three matroid $M$ is called a \emph{connected cactus} configuration if it can be obtained by inductively freely gluing lines or cycles. More precisely, there exists a sequence of matroids $N_{1},\ldots,N_{k}$ such that:
\begin{itemize}
\item $N_{1}$ is either a line or a cycle.
\item For each $i\in [k-1]$, $N_{i+1}$ is the free gluing of $N_{i}$ with a line or a cycle.
\item $N_{k}=M$.
\end{itemize}
More generally, we say that a simple rank-three matroid is a {\em cactus configuration} if each of its connected components is a connected cactus configuration.
\end{definition}

\begin{figure}
    \centering
    \includegraphics[width=0.7\linewidth]{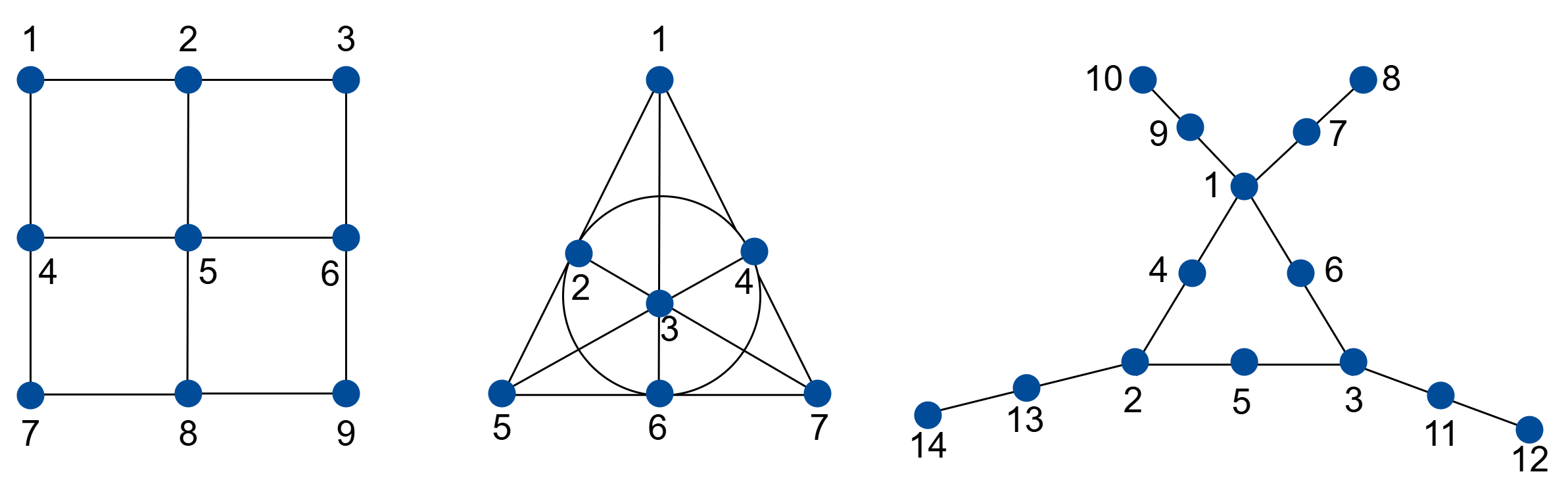}
    \caption{(Left) The $3\times 3$ grid; (Center) Fano plane; (Right) Cactus configuration.}
    \label{fig:preliminaries gemengd}
\end{figure}

An example of a cactus configuration is shown in Figure~\ref{fig:preliminaries gemengd}~(Right). Our definition of cactus configurations is closely inspired by the notion of cactus graphs in graph theory. 
A \emph{cactus} is a connected graph in which any two simple cycles share at most one vertex. Equivalently, it is a connected graph in which each edge is contained in at most one simple cycle. The following well-known property holds for cactus graphs.

\begin{lemma}\label{property}
Every cactus $G$ contains a subgraph $H$, which is either an edge or a cycle, such that at most one vertex of $H$ is adjacent to a vertex in $G\backslash H$.
\end{lemma}

We first provide an alternative characterization of cactus configurations in terms of cactus graphs.

\begin{lemma}
To each simple matroid $M$ of rank at most three, we associate a simple graph $G(M)$ defined as follows:
\begin{itemize}
\item The vertices of $G(M)$ correspond to the points in $M$ of degree at least two.
\item Two vertices of $G(M)$ are joined by an edge if and only if the corresponding points in $M$ lie on a common line. 
\end{itemize}
Then $M$ is a cactus configuration if and only if $G(M)$ is a cactus graph.
\end{lemma}

\begin{proof}
We may assume $M$ is connected, as the argument applies to each component separately.

First suppose that $G(M)$ is a cactus. We prove that $M$ is a cactus configuration by induction on the number of vertices of $G(M)$.

\medskip
{\bf Base case}: If $G(M)$ has only one vertex, then the statement is clear.

\medskip
{\bf Inductive step:} By Lemma~\ref{property}, there exists a subgraph $H$, which is either an edge or a cycle, such that at most one of its vertices is adjacent to a vertex in $G(M)\backslash H$. Let $x$ denote this vertex. We then obtain that $M$ is of the form
$M=M\rq \amalg_{p,q} N$,
where $G(M')=H$, and $G(N)=(G(M)\backslash H)\cup \{x\}$, 
and the gluing is performed at the point corresponding to $\{x\}$. Since $H$ is an edge or a cycle in $G(M)$, $M'$ is a line or a cycle configuration. Moreover, since $(G(M)\backslash H)\cup \{x\}$ is a cactus graph, the inductive hypothesis implies that $N$ is a cactus configuration. We thus conclude that $M$ is a cactus configuration.

\medskip

Conversely, suppose that $M$ is a cactus configuration. By Definition~\ref{cactus}, $M$ is constructed through a sequence of gluings of lines or cycles. We proceed by induction on the number $n$ of such gluings.

\medskip
{\bf Base case}: If $n = 0$, then $M$ consists of a single line or cycle, in which case the statement is clear.

\medskip
{\bf Inductive step}: We know that $M$ is of the form
$M=M' \amalg_{p,q} N$,
where $M'$ is a line or a cycle and $N$ is a cactus configuration. Passing to the associated graphs, we see that $G(M')$ and $G(N)$ are subgraphs of $G(M)$ sharing a unique vertex $x$ which corresponds to the point of gluing. Furthermore, $G(M')$ is an edge or a cycle of $G(M)$ and $G(N)$ is a cactus graph by inductive hypothesis. Moreover, the gluing condition ensures that there is no edge in $G(M)$ connecting a vertex in  $G(M')\backslash \{x\}$ to a vertex in $G(N)\backslash \{x\}$. Hence, $G(M)$ is also a cactus graph, completing the proof.
\end{proof}

We next prove that every cactus configuration is nilpotent; recall Definition~\ref{nil sol}.

\begin{proposition} \label{thm: cactus-like config is nilpotent}
Every cactus configuration is nilpotent.
\end{proposition}
\begin{proof}
Let $M$ be a cactus configuration. We may assume $M$ is connected, as the argument applies to each component separately.
By Definition~\ref{cactus}, $M$ is constructed through an iterative process of gluing lines or cycles. We establish the result by induction on the number $n$ of such gluings.

\medskip
{\bf Base case}: If $n = 0$, then $M$ consists of a single line or cycle, in which case the statement is clear. 

\medskip
{\bf Induction step}: Assume that the statement holds for any cactus configuration constructed from at most $n-1$ gluings, and consider a configuration $M$ obtained through $n$ gluings. At the $n^{\text{th}}$-step, $M$ takes the form
$M=M\rq \amalg_{p,q} N$,
where $M\rq$ is constructed from $n-1$ gluings, and $N$ is either a line or a cycle. Denote the ground sets of $M\rq$ and $N$ by $[d_{1}]$ and $[d_{2}]$, respectively. The ground set of $M\rq \amalg_{p,q} N$ is then given by
$([d_{1}]\backslash p)\amalg ([d_{2}]\backslash q)\cup \{P\}$,
where $P$ is the point that identifies both $p$ and $q$.

\medskip
{\bf Case~1.} Suppose that $N$ is a line, or equivalently $N=U_{2,d_{2}}$. In this case, the points in $[d_{2}]\backslash P$ have degree one in $M\rq \amalg_{p,q} N$, which implies that $S_{M}\subset M\rq$ as a submatroid. By induction, $M\rq$ is nilpotent, and since any submatroid of a nilpotent matroid is also nilpotent, it follows that $S_{M}$ is also nilpotent. Consequently, the nilpotent chain of $S_{M}$, as defined in Definition~\ref{nil sol}, eventually reaches the empty set, and the same must hold for $M$.

\medskip
{\bf Case~2.}
Now assume that $N$ is a cycle. If $q$ has degree 2 in $N$, then for every line $l$ of $N$, we have $\size{S_{M} \cap l} = 2$. Consequently, $S_{M} \cap [d_{2}]$ contains no lines, since a line must contain at least three points. This implies that $S_{S_{M}} \cap [d_{2}] \subset \{P\}$, and thus $S_{S_{M}}$ is a submatroid of $M'$. The result then follows by applying the same argument as in the previous case.
If $q$ has degree 1 in $N$, then there exists a line $l$ in $N$ for which $\size{S_{M} \cap l} = 3$, and for all other lines $l \in \mathcal{L}_N$, $\size{S_{M} \cap l} = 2$. Then, for each line $l$ of $N$, we have $\size{S_{S_M} \cap l} \leq 2$. Thus, $S_{S_M} \cap [d_2]$ contains no lines, and using the notation from Definition~\ref{nil sol}, we have $M_3 \cap [d_2] \subset \{P\}$. The result again follows by applying the same argument as in the previous case. This completes the proof.
\end{proof}

We now prove the irreducibility of matroid varieties associated with cactus configurations. 

\begin{theorem}\label{cact irr}
Every cactus configuration $M$ is realizable, and its matroid variety $V_M$ is irreducible.
\end{theorem}

\begin{proof}
By Proposition~\ref{thm: cactus-like config is nilpotent}, $M$ is nilpotent, hence the result follows from \textup{\cite[Theorem~4.12]{liwski2024pavingmatroidsdefiningequations}}.
\end{proof}

\subsection{Matroid ideal of cactus configurations}

In this subsection, we present a complete set of defining equations for the matroid varieties associated with cactus configurations, or equivalently, a finite generating set for their matroid ideals up to radical. To achieve this, we first establish a series of lemmas.

\begin{definition}
Let $M$ be a matroid on $[d]$, and consider a subset $X \subset [d]$. We say that $X$ contains a cycle if there exist distinct points $x_1, \ldots, x_k \in X$ and distinct lines $l_1, \ldots, l_k$ of $M$ such that:
\begin{itemize}
    \item $\{x_i, x_{i+1}\} \subset l_i$ for every $i \in [k]$, where $x_{k+1} = x_1$.
\end{itemize}
\end{definition}

\begin{lemma}\label{Generalization of Lemma 4.21 i}
Let $M$ be a simple rank-three matroid on $[d]$ and let $X \subset [d]$ be a subset of points that does not contain a cycle. Then, there exists a point $x \in X$ such that 
\[|\{l \in \mathcal{L}_x: l \cap (X \backslash \{x\}) \neq \emptyset\}| \leq 1.\]
\end{lemma}

\begin{proof}
Suppose the contrary. Choose an arbitrary point $x_{1} \in X$. By assumption, there exists a line $l_{1} \in \mathcal{L}_{x_{1}}$ containing another point $x_{2} \in X$ with $x_{2} \neq x_{1}$. Applying the hypothesis again, we find a distinct line $l_{2} \neq l_{1}$ that contains $x_{2}$ and another point $x_{3} \neq x_{2}$. Repeating this process, we construct sequences of points $x_{1}, x_{2}, \ldots$ and lines $l_{1}, l_{2}, \ldots$ such that $x_{k} \neq x_{k+1}$, $l_{k} \neq l_{k+1}$, and $\{x_k, x_{k+1}\} \subset l_k$. Since $X$ is finite, some point must eventually repeat, forming a cycle. This contradicts the assumption that $X$ does not contain a cycle, completing the proof.
\end{proof}

Recall from Definition~\ref{nil sol} that $Q_{M}$ denotes the set of points in $M$ with degree at least three.

\begin{lemma}\label{Generalization of Lemma 4.21 ii}
Let $M$ be a cactus configuration, and let $\gamma \in V_{\mathcal{C}(M)}$ with $\gamma_p \neq 0$ for all $p \in Q_M$. Then $\gamma$ can be perturbed to obtain a configuration $\tau \in \Gamma_M$.
\end{lemma}

\begin{proof}
We may assume that $M$ is connected, since the argument can then be applied separately to each connected component. By Definition~\ref{cactus}, $M$ is constructed through a sequence of gluings of lines or cycles. We prove the claim by induction on the number $m$ of such gluings.

\medskip
{\bf Base case:} If $m = 0$, then $M$ is either a line or a cycle. By Theorem~\ref{nil coincide} (i), it follows that $\VCM = V_M$, so we can perturb $\gamma$ to obtain $\tau \in \Gamma_M$.

\medskip
{\bf Induction step}: Assume that the statement holds for any cactus configuration constructed from at most $m-1$ gluings, and consider a cactus configuration $M$ obtained through $m$ gluings. At the $m^{\text{th}}$-step, $M$ takes the form
$M=M\rq \amalg_{p,q} N$,
where $M\rq$ is constructed from $m-1$ gluings, and $N$ is either a line or a cycle. Denote the ground sets of $M\rq$ and $N$ by $[d_{1}]$ and $[d_{2}]$, respectively. The ground set of $M\rq \amalg_{p,q} N$ is then 
$([d_{1}]\backslash p)\amalg ([d_{2}]\backslash q)\cup \{P\}$,
where $P$ is the point that identifies both $p$ and $q$.

\medskip
We first show that we may assume $\gamma_{P}\neq 0$. 
\begin{itemize}
\item If $P\in Q_{M}$, then $\gamma_{P}\neq 0$ by assumption. 
\item If $|\mathcal{L}_P| \leq 2$ and $\gamma_{P}=0$, consider the set $R =\{\gamma_l:l \in \mathcal{L}_i \textup{ and } \rk(\gamma_l)=2\}$. Since $i$ has degree at most two, we have $|R| \leq 2.$ If $|R| = 0$, then perturb $\gamma$ arbitrarily away from the origin. If $|R|=1$, then perturb $\gamma$ on the unique line in $R$, away from the origin. If $|R|=2$, then perturb $\gamma$ on the intersection of the two distinct lines in $R$. In each case, the resulting configuration lies in $\Gamma_M$, as desired.
\end{itemize}

{\bf Case~1.} Suppose that $N$ is a line. By the induction hypothesis, we can perturb the vectors $\{\gamma_r:r\in [d_{1}]\}$ to obtain a collection $\{\tau_r:r\in [d_{1}]\}$ in $\Gamma_{M\rq}$. 
Since $\gamma_{P}\neq 0$, we can further extend $\tau$ to the points in $[d_{2}] \backslash \{q\}$ by applying a perturbation to the vectors $\{\gamma_{r}:r\in [d_{2}]\backslash q\}$, ensuring that these vectors remain on the same line as $\tau_P$ and on no other line. This yields a collection of vectors $\tau \in \Gamma_{M}$, as desired.

\medskip
{\bf Case~2.} Suppose that $N$ is a cycle. By the induction hypothesis, we can perturb the vectors $\{\gamma_r : r \in [d_{1}]\}$ to obtain a collection $\{\tau_r : r \in [d_{1}]\}$ in $\Gamma_{M'}$.  
Similarly, as in the first case, we can perturb the vectors $\{\gamma_r : r \in [d_{2}]\}$ to obtain a collection $\{\beta_r : r \in [d_{2}]\}$ in $\Gamma_{N}$.  
Furthermore, since $\gamma_p \neq 0$, by applying a small rotation from the origin to the vectors in $\beta$, we can ensure that $\beta_q = \tau_p$.  
Next, we rotate $\beta$ infinitesimally around $\tau_p$ and observe that each $\beta_r \neq 0$ since $\beta \in \Gamma_N$. This guarantees that the vectors $\beta_r$ do not lie on any lines of the configuration other than those associated with $N$, for all $r \in [d_{2}]$.  
This yields a collection of vectors in $\Gamma_M$, as desired.
\end{proof}

\begin{lemma}\label{Generalization of Lemma 4.21 iii}
Let $M$ be a cactus configuration, such that the points of $Q_M$ do not contain a cycle. If $\gamma \in V_{\mathcal{C}(M)} \cap V(G_{M})$, then there exists $\tau \in V_{\mathcal{C}(M)}$ such that:
\begin{itemize}
\item For every point $p \in Q_{M}$, we have $\tau_{p} \neq 0$.
\end{itemize}
In particular, $\tau$ can be chosen as a perturbation of $\gamma$.
\end{lemma}

\begin{proof}
The proof follows by applying the same argument as in \textup{\cite[Lemma~4.23]{liwski2024pavingmatroidsdefiningequations}}.
\end{proof}

We now establish the main result of this section.

\begin{theorem} \label{thm: main theorem cactus matroid ideal}
   Let $M$ be a cactus configuration, such that the points of $Q_M$ do not contain a cycle. Then $I_M = \sqrt{I_{\mathcal{C}(M)} +G_M}$. 
\end{theorem}
\begin{proof}
Denote by $I$ the ideal on the right-hand side, which is contained in $I_M$ by Proposition \ref{inc G}. To show the reverse inclusion, it suffices to show $V(I) \subseteq V_M$. Consider a collection of vectors $\gamma\in \VCM \cap V(G_M)$. By Lemma~\ref{Generalization of Lemma 4.21 iii}, we can perturb the vectors of $\gamma$ to a collection of vectors $\tau \in \VCM$ such that $\tau_p \neq 0$ for all $p \in Q_M$. Applying Lemma~\ref{Generalization of Lemma 4.21 ii}, we obtain a collection of vectors $\tau'\in \Gamma_{M}$ which is a perturbation of $\tau$. This proves that $\gamma \in V_{M}$, as desired.
\end{proof}

In the following example, we show that the acyclicity assumption on the points of $Q_M$ in Theorem~\ref{thm: main theorem cactus matroid ideal} is essential and cannot be omitted.

\begin{example}
Consider the matroid depicted in Figure~\ref{fig:preliminaries gemengd} (Right), which we denote by $M$.
    Let 
$$\gamma = \begin{pmatrix}
    1 & 0 & 0 & 1 & 1 & 1 & 1 & 1 & 1 & 1 & 1 & 0 & 0 & 0  \\
    0 & 1 & 0 & 1 & 2 & 4 & 5 & 6 & 7 & 9 & 8 & 0 & 0 & 0 \\
    0 & 0 & 1 & 1 & 3 & 8 & 7 & 10 & 12 & 21 & 17 & 0 & 0 & 0
        \end{pmatrix},$$
where the columns correspond to the points $\{4,5,6, 7,8,9,10,11,12,13,14,1,2,3\}$, respectively. Note that $\gamma \in \VCM \cap V(G_M)$. We will show that $\gamma\not \in V_{M}$.

Assume, for contradiction, that $\gamma \in V_M = \overline{\Gamma_M}$. Since the Zariski and Euclidean closures of $\Gamma_M$ coincide, it follows that for every $\epsilon > 0$, we can choose an element $\gamma_\epsilon \in \Gamma_M$ such that $\lvert \gamma - \gamma_\epsilon \rvert < \epsilon$. After a perturbation, we may assume that $\gamma_\epsilon$ takes the form
{\small
\[
\gamma_{\epsilon} = \begin{pmatrix}
    1 & 0 & 0 & 1 & 1 & 1 & 1 & 1 & 1 & 1 & 1 & x_1 & x_2 & x_3  \\
    0 & 1 & 0 & 1 & 2 + \epsilon(1) & 4 + \epsilon(3) & 5 + \epsilon(5) & 6 + \epsilon(7) & 7 + \epsilon(9) & 9 + \epsilon(11) & 8 + \epsilon(13) & y_1 & y_2 & y_3 \\
    0 & 0 & 1 & 1 & 3 + \epsilon(2) & 8 + \epsilon(4) & 7 + \epsilon(6) & 10 + \epsilon(8) & 12 + \epsilon(10) & 21 + \epsilon(12) & 17 + \epsilon(14) & z_1 & z_2 & z_3
\end{pmatrix},
\]
}
where the columns again correspond to $\{4,5,6, 7,8,9,10,11,12,13,14,1,2,3\}$ and $\epsilon(i)\in \CC$ satisfies $\lvert \epsilon(i)\rvert <\epsilon$ for all $i\in [14]$.
Note that $\{1,9,10\}$ and $\{1,7,8\}$ are lines in $M$, hence the point ${\gamma_{\epsilon}}_{1}$ lies in the intersection of the lines spanned by $\{{\gamma_{\epsilon}}_{9},{\gamma_{\epsilon}}_{10}\}$ and $\{{\gamma_{\epsilon}}_{7},{\gamma_{\epsilon}}_{8}\}$. Therefore, 

\[{\gamma_{\epsilon}}_{1}={\gamma_{\epsilon}}_{9}{\gamma_{\epsilon}}_{10}\wedge {\gamma_{\epsilon}}_{7}{\gamma_{\epsilon}}_{8}.\]
Taking the limit $\epsilon \to 0$, we obtain
$\lim_{\epsilon \to 0} {\gamma_{\epsilon}}_{1}=\gamma_{9}\gamma_{10}\wedge \gamma_7 \gamma_8=(1,13/3,23/3)^{\top}$.

Similarly, $\{3,1,6\}$ and $\{3,11,12\}$ are lines of $M$, hence ${\gamma_{\epsilon}}_{3}$ lies in the intersection of the lines spanned by $\{{\gamma_{\epsilon}}_{1},{\gamma_{\epsilon}}_{6}\}$ and $\{{\gamma_{\epsilon}}_{11},{\gamma_{\epsilon}}_{12}\}$. Thus,
${\gamma_{\epsilon}}_{3}={\gamma_{\epsilon}}_{1}{\gamma_{\epsilon}}_{6}\wedge {\gamma_{\epsilon}}_{11}{\gamma_{\epsilon}}_{12}$.
Taking the limit $\epsilon \to 0$, we obtain
\[\lim_{\epsilon \to 0} {\gamma_{\epsilon}}_{3}=\lim_{\epsilon \to 0}{\gamma_{\epsilon}}_{1}{\gamma_{\epsilon}}_{6}\wedge {\gamma_{\epsilon}}_{11}{\gamma_{\epsilon}}_{12}=(1,13/3,23/3)^{\top}\gamma_{6}\wedge \gamma_{11} \gamma_{12}=(1,13/3,20/3)^{\top}. \]

Similarly, $\{2,3,5\}$ and $\{2,13,14\}$ are lines of $M$. Thus, the triples of points $\{{\gamma_{\epsilon}}_{2},{\gamma_{\epsilon}}_{3},{\gamma_{\epsilon}}_{5}\}$ and $\{{\gamma_{\epsilon}}_{2},{\gamma_{\epsilon}}_{13},{\gamma_{\epsilon}}_{14}\}$ are collinear. Consequently, 
${\gamma_{\epsilon}}_{2}={\gamma_{\epsilon}}_{3}{\gamma_{\epsilon}}_{5}\wedge {\gamma_{\epsilon}}_{13}{\gamma_{\epsilon}}_{14}$.
Taking the limit $\epsilon \to 0$, we obtain
\[\lim_{\epsilon \to 0} {\gamma_{\epsilon}}_{2}=\lim_{\epsilon \to 0}{\gamma_{\epsilon}}_{3}{\gamma_{\epsilon}}_{5}\wedge {\gamma_{\epsilon}}_{13}{\gamma_{\epsilon}}_{14}=(1,13/3,20/3)^{\top}\gamma_{5}\wedge \gamma_{13} \gamma_{14}=(1,65/12,80/12)^{\top}. \]

Now, since $\{1,2,4\}$ is a line of $M$, and $\gamma_{\epsilon}\in \Gamma_{M}$, it must hold that $\det({\gamma_{\epsilon}}_{1},{\gamma_{\epsilon}}_{2},{\gamma_{\epsilon}}_{4})=0$.
However, evaluating the limits give
\begin{equation*}
\begin{aligned}
\lim_{\epsilon \to 0}\det({\gamma_{\epsilon}}_{1},{\gamma_{\epsilon}}_{2},{\gamma_{\epsilon}}_{4})&=\det(\lim_{\epsilon \to 0}{\gamma_{\epsilon}}_{1},\lim_{\epsilon \to 0}{\gamma_{\epsilon}}_{2},\gamma_{4})\\
&=\det((1,13/3,23/3)^{\top},(1,65/12,80/12)^{\top},\gamma_{4})\\
&=-455\neq 0,
\end{aligned}
\end{equation*}
a contradiction. Thus $\gamma \notin V_{M}$, and so the acyclicity assumption in Theorem~\ref{thm: main theorem cactus matroid ideal} is~necessary.
\end{example}

\subsection{Bound on the number of irreducible components of $\VCM$}
We now bound the number of irreducible components of the  circuit variety of a cactus configuration.

\begin{theorem} \label{decomposition cactus}
    Let $M$ be a cactus configuration. Then $\VCM$ has at most $2^{\lvert Q_{M} \rvert}$ irreducible components, each arising from setting a subset of $Q_{M}$ to be loops.
\end{theorem}

\begin{proof}
Let $[d]$ be the ground set of $M$. For each subset $J\subset Q_{M}$, let $M(J)$ denote the matroid obtained from $M$ by setting the points of $J$ to be loops.
We claim that the following equality holds:
\begin{equation}\label{vcm}
V_{\mathcal{C}(M)}=\bigcup_{J\subset Q_{M}}V_{M(J)}.
\end{equation}
The inclusion $\supset$ is clear since the matroids $M(J)$ have more dependencies than $M$. To establish the reverse inclusion, we must verify that every $\gamma\in V_{\mathcal{C}(M)}$ lies in the variety on the right-hand side of~\eqref{vcm}. To verify this, let us fix an arbitrary element $\gamma\in V_{\mathcal{C}(M)}$.
Let $X$ be the set
$\{x \in Q_M : \gamma_x = 0\}$.
The submatroid $M\backslash X$ is also a cactus configuration. Moreover, since $\gamma_{p}\neq 0$ for all $p\in [d]\backslash X$, it follows that $\gamma_{p}\neq 0$ for all $p\in Q_{M\backslash X}$. Therefore, by Lemma~\ref{Generalization of Lemma 4.21 ii}, we know that $\restr{\gamma}{[d]\backslash X}\in V_{M\backslash X}$, which implies that $\gamma\in V_{M(X)}$. This proves the claim. 

By Theorem~\ref{cact irr}, each variety on the right-hand side of~\eqref{vcm} is irreducible. It follows that $V_{\mathcal{C}(M)}$ has at most $2^{\size{Q_{M}}}$ irreducible components.
\end{proof}

\section{
Pascal and Pappus configurations}\label{section pascal and papus}

In this section, we determine a finite generating set, up to radical, for the matroid ideals of the Pascal and Pappus configurations.

\subsection{A sufficient criterion for liftability}

The main result of this subsection is Proposition~\ref{prop: nilpotent add point liftable}, which establishes a sufficient condition for an affirmative answer to the following question:

\begin{question}\normalfont\label{question}
Let $M$ be a simple rank-three matroid on $[d]$ and suppose $\{\gamma_{1},\ldots,\gamma_{d-1}\}\subset \mathbb{P}^{2}$ is a collection of points lying on a common line $l\subset \mathbb{P}^{2}$ which is liftable from a point $q\in \mathbb{P}^{2}$ to a collection in $V_{\CCC(M\backslash d)}$. 
Does there exist a point $\gamma_{d}\in l$ such that the extended collection $\{\gamma_{1},\ldots,\gamma_{d}\}$ is liftable from $q$ to a collection in $V_{\CCC(M)}$?  
\end{question}

In the following sections, we will apply Proposition~\ref{prop: nilpotent add point liftable} to compute the matroid ideals associated with the Pascal and Pappus configurations.
In the context of the preceding question and throughout this subsection, a lifting always refers to a non-degenerate one, that is, a lifting in which the resulting vectors span the entire space. We begin by introducing a definition.

\begin{definition}\label{definition dim lift}
    Let $M$ be a simple rank-three matroid on $[d]$ and let $\gamma$ be a collection of vectors in $V_{\CCC(M)}$ and $q\in \CC^{3}$ a vector in general position with respect to $\gamma$. Following Definition~\ref{def liftable}, we denote by $\Lift_{M,q}(\gamma)\subset \CC^{d}$ the subspace composed of all vectors $(z_{1},\ldots,z_{d})$ such that the collection of vectors given by $\widetilde{\gamma}_{i}=\gamma_{i}+z_{i}q$ belongs to $V_{\CCC(M)}$. We also denote
    \[\dimm{q}{\gamma}(M)=\dim(\Lift_{M,q}(\gamma)),\]
the dimension of the space of liftings of $\gamma$. By~\cite[Lemma~4.31]{liwski2024pavingmatroidsdefiningequations}, this number equals
$\dim(\ker(\liftmat))$,
where $\liftmat$ is the matrix defined in Definition~\ref{matrix lift}. For any submatroid $M|S$, we denote by $\dimm{q}{\gamma}(M|S)$ the dimension
$\dimm{q}{\gamma}(M|S) := \dim\big(\Lift_{M|S,q}(\restr{\gamma}{S})\big)$.

\end{definition}

We now show that, under an additional natural assumption, the quantity $\dimm{q}{\gamma}(M)$ is independent of the choice of $q$ and $\gamma$ for nilpotent matroids. Before that we introduce some notation.

\begin{notation}
Let $M$ be a matroid on $[d]$ and let $w=(p_{1},\ldots,p_{d})$ be an ordering of its elements. We denote by $M_{w,i}$ the restriction $M\mid \{p_{1},\ldots,p_{i}\}$, and we define $w_{i}$ to be the degree of $p_{i}$ in $M_{w,i}$.
Note that if $M$ is nilpotent, then there exists an ordering $w$ for which
\begin{equation}\label{max}\text{max}\{w_{i}:i\in [d]\}\leq 1.\end{equation}
\end{notation}

\begin{theorem}\label{qgamma is constant}
Let $M$ be a nilpotent matroid on $[d]$
and let $w=(i_{1},\ldots,i_{d})$ be an ordering of its points satisfying condition~\eqref{max}. Then, for any $\gamma\in V_{\CCC(M)}$ with no coinciding points or loops, and for any $q\in \CC^{3}$ in general position with respect to $\gamma$, we have
 \begin{equation}\label{suma wi}\dimm{q}{\gamma}(M)=d-\sum_{i=1}^{d}w_{i}.\end{equation}
\end{theorem}

\begin{proof}
We proceed by induction on $d$, the size of the ground set. The base case $d=3$ is immediate.

For the inductive step, assume the result holds for all nilpotent matroids with at most $d-1$ elements. Let $M$ be a nilpotent matroid on $[d]$, and consider an arbitrary collection of vectors $\gamma \in V_{\mathcal{C}(M)}$ with no coinciding points or loops. Let $q \in \mathbb{C}^3$ be in general position with respect to $\gamma$. Our goal is to establish~\eqref{suma wi}. Let $\gamma'$ be the restriction of $\gamma$ to $\{p_1, \ldots, p_{d-1}\}$. By the inductive hypothesis, we have
\begin{equation}\label{suma wi 2}\dim_{q}^{\gamma'}(M_{w,d-1})=d-1-\sum_{i=1}^{d-1}w_{i}.\end{equation}
We consider two cases based on the value of $w_d = \deg(p_d)$.

\medskip
\textbf{Case 1.} Suppose that $\deg(p_{d}) = 0$. In this case, the point $p_{d}$ does not belong to any circuit of size three in $M$. Therefore, any lifting of the vectors in $\gamma$ is completely determined by a lifting of the vectors in $\gamma'$, together with an arbitrary lifting of the vector $\gamma_{p_{d}}$. Hence,
$\dimm{q}{\gamma}(M) = \dim_{q}^{\gamma'}(M_{w,d-1}) + 1$.
Substituting Equation~\eqref{suma wi 2}, we obtain Equation~\eqref{suma wi}, as desired.

\medskip
\textbf{Case~2.} Suppose that $\deg(p_{d})=1$. Then $p_{d}$ lies on a unique line $l$ of $M$. Any lifting of the vectors $\gamma$ is thus determined by a lifting of the vectors $\gamma'$, together with a unique scalar $z_{p_{d}}$ such that the lifted vector
$\gamma_{p_{d}} + z_{p_{d}} q$
belongs to the two-dimensional subspace
$\mathrm{span}\{\gamma_i : i \in l \setminus \{p_{d}\}\}$.
This subspace is two-dimensional because $\gamma \in V_{\CCC(M)}$ has no coinciding points, and the uniqueness of $z_{p_{d}}$ follows from the assumption that $\gamma$ has no loops. Therefore, the contribution from $p_{d}$ does not increase the dimension, and we have
$\dim_q^\gamma(M) = \dim_q^{\gamma'}(M_{w,d-1})$.
Applying Equation~\eqref{suma wi 2} again yields Equation~\eqref{suma wi}.
\end{proof}

\begin{definition}
For a nilpotent matroid $M$, we denote by $\dim(M)$ the number from Equation~\eqref{suma wi}.
\end{definition}

The next proposition provides an affirmative answer to Question~\ref{question} in the case where the degree of the point $d$ is at most two. 

\begin{proposition} \label{prop: degree < 3 add point liftable}
Let $M$ be a simple rank-three matroid on $[d]$ and suppose $\{\gamma_{1},\ldots,\gamma_{d-1}\}\subset \mathbb{P}^{2}$ is a collection of points lying on a common line $l\subset \mathbb{P}^{2}$ which is liftable from a point $q\in \mathbb{P}^{2}$ to a collection in $V_{\CCC(M\backslash d)}$. 
Moreover, assume that $\deg(d)\leq 2$. Then, there exists a point $\gamma_{d}\in l$ such that the extended collection $\{\gamma_{1},\ldots,\gamma_{d}\}$ is liftable from $q$ to a collection in $V_{\CCC(M)}$.
\end{proposition}

\begin{proof}
We divide the proof into cases depending on the degree of $d$.
\medskip

{\bf Case~1.} Suppose that $\deg(d)=0$. In this case, we can extend $\gamma$ by selecting an arbitrary point $\gamma_{d}\in l$. The resulting collection is liftable from $q$ since any lifting of $\{\gamma_{1},\ldots,\gamma_{d-1}\}$ to a collection in $V_{\CCC(M\backslash \{d\})}$ and an arbitrary lifting of the point
$\gamma_{d}$ represents a lifting to a collection in $V_{\CCC(M)}$.

\medskip
{\bf Case~2.}
Suppose that $\deg(d) = 1$ and let $l$ be the unique line of $M$ to which it belongs. In this case, again any arbitrary point $\gamma_{d}\in l$ results in a collection which is liftable from $q$. To see this, consider the lifting of the points $\{\gamma_{1},\ldots,\gamma_{d-1}\}$ from $q$ to a collection $\widetilde{\gamma}$ in $V_{\CCC(M\backslash d)}$  and lift the point $\gamma_{d}$ from $q$ to a point in $\widetilde{\gamma}_{l}$ resulting in a non-degenerate collection in $V_{\CCC(M)}$.

\medskip
{\bf Case~3.} Suppose that $\deg(p) = 2$ and let $l_{1}$ and $l_{2}$ be the lines of $M$ containing $d$.
In this case, if $\widetilde{\gamma}\in V_{\CCC(M\backslash d)}$ is  a lifting of the points $\{\gamma_{1},\ldots,\gamma_{d-1}\}$ from $q$, extending $\gamma$ by setting $\gamma_d$ as the projection of $\widetilde{\gamma}_{l_{1}} \cap \widetilde{\gamma}_{l_{2}}$ from $q$ on $l$ works by applying the same argument as in Case~2.
\end{proof}

The preceding proposition essentially establishes that if $\deg(d)\leq 2$, then Question~\ref{question} has a positive answer. We now provide a sufficient condition for obtaining a positive answer when $\deg(d)\geq 3$. Before introducing the final property of this section, we recall a classical result from algebraic geometry, stated as Proposition~I.7.1 in \cite{Hartshorne}.

\begin{proposition} \label{prop: hartsorne}
    Let $Y\subset \CC^{n}$ and $Z\subset \CC^{n}$ be varieties of dimensions $r$ and $s$, respectively. Then every irreducible component of $Y\cap Z$ has dimension at least $r+s-n.$
\end{proposition}

We now prove the main result of this subsection, which will be fundamental in what follows.

\begin{proposition}\label{prop: nilpotent add point liftable}
Let $M$ be a simple rank-three matroid on $[d]$. Then, within the framework of Question~\ref{question}, the answer is affirmative under the following assumptions:
\begin{enumerate}
\item $M \backslash \{d\}$ is nilpotent and satisfies $\dim(M \backslash \{d\}) \geq 1+\deg(d).$
\item $\gamma$ has no coinciding points.
\end{enumerate} 
\end{proposition}

\begin{proof}
Let $\mathcal{L}_{d}$ denote the set of lines of $M$ containing the point $d$.
We will first prove that there exists a vector $z=(z_{1},\ldots,z_{d})$ such that the
lifting $\tau_{z}=\{\gamma_{1}+z_{1}q,\ldots,\gamma_{d-1}+z_{d-1}q\}$ of the points $\{\gamma_{1},\ldots,\gamma_{d-1}\}$ satisfies the following conditions:
\begin{enumerate}
\item $\tau_{z}\in V_{\CCC(M\backslash \{d\})}$.
\item The lines $\{(\tau_{z})_{l}:l\in \mathcal{L}_{d}\}$ are concurrent.
\item $\tau_{z}$ is non degenerate.
\end{enumerate}
We define two varieties $Y,Z\subset \CC^{d}$ as follows:
\[Y=\{z\in \CC^{d}:\tau_{z}\in V_{\CCC(M\backslash \{d\})}\}, \quad \text{and} \quad Z=\{z\in \CC^{d}:\text{the lines $\{(\tau_{z})_{l}:l\in \mathcal{L}_{d}\}$ are concurrent}\}.\]
By definition, $Y$ is precisely the linear space $\Lift_{M\backslash \{d\},q}(\gamma)$ introduced in Definition~\ref{definition dim lift}, and thus
\[\dim(Y)=\dim_{q}^{\gamma}(M\backslash \{d\})=\dim(M\backslash \{d\}),\]
where the final equality follows from Theorem~\ref{qgamma is constant}.

To describe $Z$, let $\{l_{1},\ldots,l_{\deg(d)}\}$ be the lines in $\mathcal{L}_{d}$, and for each $l_{i}$, choose two distinct points $p_{i_{1}}$ and $p_{i_{2}}$ on $l_{i}$. Then, $Z$ consists of those $z\in \CC^{d}$ for which the following families of lines are concurrent:
\begin{equation*}
    \begin{array}{ll}
       \{(\tau_{z})_{l_{1}},(\tau_{z})_{l_{2}},(\tau_{z})_{l_{3}}\}
        , 
        \ldots,  
       \{(\tau_{z})_{l_{1}},(\tau_{z})_{l_{2}},(\tau_{z})_{l_{\deg(d)}}\}.
    \end{array}
\end{equation*}
Using the Grassmann-Cayley algebra, we have that the concurrency of these $\text{deg}(d)-2$ triples of lines can be translated as the vanishing of the following $\text{deg}(d)-2$ polynomials:
\[((\tau_{z})_{p_{1_{1}}} \vee (\tau_{z})_{p_{1_{2}}}) \wedge ((\tau_{z})_{p_{2_{1}}} \vee (\tau_{z})_{p_{2_{2}}}) \vee ((\tau_{z})_{p_{j_{1}}} \vee (\tau_{z})_{p_{j_{2}}}),\]
for $3\leq j\leq d$. Since $Z$ is given by the vanishing locus of $\deg(d)-2$ polynomials, we have 
$\dim(Z)\geq d+2-\deg(d)$.
Note that $0\in Y\cap Z$, so $Y\cap Z$ is nonempty. By Proposition~\ref{prop: hartsorne}, 
\[\dim(Y\cap Z)\geq \dim(Y)+\dim(Z)-d\geq \dim(M\backslash \{d\})+2-\deg(d).\]
Since $\dim(M\backslash \{d\})\geq \deg(d)+1$ by hypothesis, it follows that
$\dim(Y\cap Z)\geq 3$.
Since the space of trivial liftings of $\gamma$, corresponding to degenerate liftings, has dimension two, we then conclude that there exists $z\in Y\cap Z$ such that $\tau_{z}$ is non-degenerate. This lifting satisfies all the desired conditions.

Finally, let $x$ be the unique point of intersection of the lines $\{(\tau_{z})_{l} : l \in \mathcal{L}_{d}\}$. We define $\gamma_{d}$ as the projection of $x$ from $q$ onto the line $l$. This yields a liftable collection with respect to $M$ and $q$.
\end{proof}

\subsection{Pascal Configuration}

In this section, we study the Pascal configuration, the simple rank-three matroid in Figure~\ref{fig:pascal}(b). Specifically, we determine the irreducible decomposition of its circuit variety and find a finite generating set for its matroid ideal, up to radical. Throughout this subsection, we denote this matroid by $M$.

\subsubsection{Irreducible decomposition of $V_{\mathcal{C}(M)}$}

We introduce the following notation for the remainder of the paper:
\begin{notation}\label{notation pi}
For a matroid $N$ of rank at most three on $[d]$, we define:
\begin{itemize}
\item $\pi_N^i$ as the matroid on $[d]$ in which all points, except for $i$, lie on a single line. Moreover, points that share a common line with $i$ in $N$ are identified. See Figure~\ref{fig:pascal} (a) for an illustration of $\pi_{M}^{1}$.
\item $N(i)$ denotes the matroid obtained by making $i$ a loop. More generally, for a subset $\{i_{1},\ldots,i_{k}\}\subset [d]$, we define $N(i_{1},i_{2}, \ldots, i_{k})$ as the matroid obtained by turning all points in $\{i_{1},i_{2}, \ldots, i_{k}\}$ into loops. This matroid is isomorphic to $N\backslash \{i_{1},\ldots,i_{k}\}$ up to removal of loops.
\item Given a collection of vectors $\{\gamma_{1},\ldots,\gamma_{d}\}\subset \CC^{3}$ and a subset $A\subset [d]$, we denote by $\gamma_{A}\subset \CC^{3}$ the linear span of the vectors $\{\gamma_{a}:a\in A\}$.
    \end{itemize}
\end{notation}

Moreover, the following remark will be used throughout the remainder of the paper.

\begin{remark}
For any given configuration, there exists a sufficiently small perturbation that does not introduce new dependencies. Throughout, we assume that all perturbations satisfy this condition.
\end{remark}

The following lemmas 
plays a key role in our discussion. Recall that for a matroid $N$, we denote by $N(i)$ the matroid obtained from $N$ by declaring $i$ to be a loop.

\begin{lemma} \label{lemma: third config loop matroid variety}
    Let $M$ be a simple rank-three matroid on $[d]$ and assume that the point $i\in [d]$ has degree at most two. Then $V_{M(i)} \subseteq V_M$.
\end{lemma}

\begin{proof}
We show that any collection $\gamma \in \Gamma_{M(i)}$ can be perturbed to a collection in $\Gamma_{M}$, hence $\gamma\in V_{M}$. Let $\gamma \in \VCM$, and consider the set $R =\{\gamma_l:l \in \mathcal{L}_i \textup{ and } \rk(\gamma_l)=2\}$. Since $i$ has degree at most two, we have $|R| \leq 2.$ If $|R| = 0$, then perturb $\gamma$ arbitrarily away from the origin. If $|R|=1$, then perturb $\gamma$ on the unique line in $R$, away from the origin. If $|R|=2$, then perturb $\gamma$ on the intersection of the two distinct lines in $R$. In each case, the resulting configuration lies in $\Gamma_M$, as desired.
\end{proof}

\begin{figure}[H]
    \centering
    \includegraphics[width=0.55\linewidth]{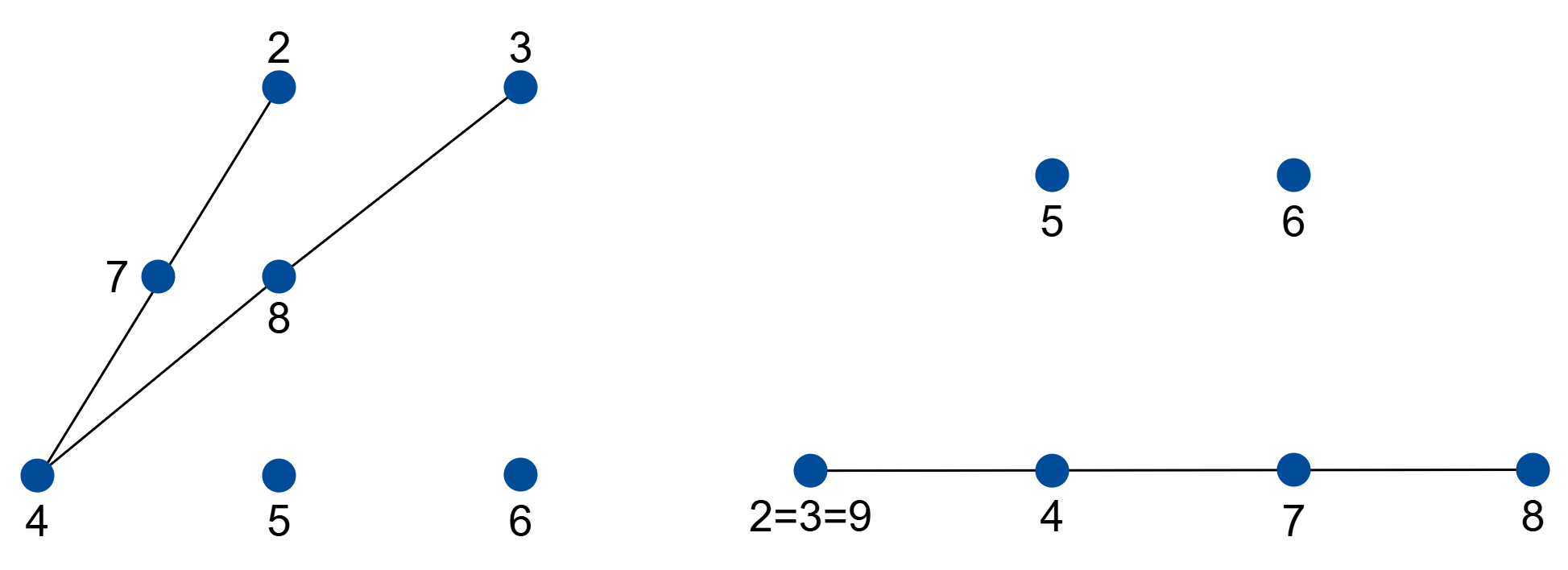}
    \caption{(Left) Matroid $N \backslash \{9\}$; (Right) Matroid $A$. Both arise in the proof of Lemma~\ref{props}.}
    \label{fig:appendix}
\end{figure}

\begin{lemma}\label{props}
    Let $N$ be the simple rank-three matroid depicted in Figure~\ref{fig:pascal} (c). Then, we have \begin{equation}\label{n4}V_{\mathcal{C}(N)} = V_{N} \cup V_{N(9)}.\end{equation}
\end{lemma}
\begin{proof} 
Let $\gamma \in \VCN$. We will prove that $\gamma$ belongs to one of the varieties on the right-hand side of~\eqref{n4}. First, suppose that $\gamma_{9} = 0$. Since $N \backslash \{9\}$, shown in Figure~\ref{fig:appendix} (Left), is nilpotent with all points of degree at most two, Theorem~\ref{nil coincide} (i) implies that $V_{\CCC(N \backslash \{9\})} = V_{N \backslash \{9\}}$, which further implies $V_{\CCC(N(9))} = V_{N(9)}$. From this, we conclude that $\gamma \in V_{N(9)}$.

Now suppose that $\gamma_9 \neq 0$. Then, by Lemma~\ref{lemma: third config loop matroid variety}, we may assume that $\gamma$ has no loops.

\medskip
{\bf Case~1.} Suppose that $\gamma_2 \neq \gamma_9$ or $\gamma_3 \neq \gamma_9$. Without loss of generality, assume that $\gamma_2 \neq \gamma_9$. By Theorem~\ref{nil coincide} (i), we can apply a perturbation to the vectors $\{\gamma_2, \gamma_3, \gamma_4, \gamma_5, \gamma_7, \gamma_8, \gamma_9\}$ to obtain a collection $\{\widetilde{\gamma}_2, \widetilde{\gamma}_3, \widetilde{\gamma}_4, \widetilde{\gamma}_5, \widetilde{\gamma}_7, \widetilde{\gamma}_8, \widetilde{\gamma}_9\} \in \Gamma_{N \backslash \{6\}}$.
Since $\gamma_2 \neq \gamma_9$, the line spanned by $\widetilde{\gamma}_2$ and $\widetilde{\gamma}_9$ is a perturbation of the line through $\gamma_2$ and $\gamma_9$. Since the only line in $N$ containing the point $6$ is $\{2,9,6\}$, we can perturb $\gamma_6$ so that it lies on the line through $\widetilde{\gamma}_2$ and $\widetilde{\gamma}_9$. This yields a collection of vectors in $\Gamma_N$, showing that $\gamma \in V_N$.

\medskip
{\bf Case~2.} Suppose that $\gamma_2= \gamma_3= \gamma_9$. Since $\{2,7,4\},\{7,8,9\}$ and $\{3,4,8\}$ are dependent in $N$ and $\gamma\in V_{\CCC(N)}$, the sets of vectors $\{\gamma_2,\gamma_7,\gamma_4\},\{\gamma_7,\gamma_8,\gamma_9\}$ and $\{\gamma_3,\gamma_4,\gamma_8\}$ are dependent. Moreover, using that $\gamma_2= \gamma_3= \gamma_9$, we obtain that the sets of vectors $\{\gamma_2,\gamma_7,\gamma_4\}, \{\gamma_2,\gamma_7,\gamma_8\}, \{\gamma_2,\gamma_4,\gamma_8\}$ are dependent.
Hence, $\{\gamma_2,\gamma_4,\gamma_7,\gamma_8\}$ are collinear in $\mathbb{P}^{2}$.

Let $A$ be the rank-three matroid on the ground set $\{2,\ldots,9\}$, in which $2=3=9$ and the points $\{2,4,7,8\}$ form a line with the points $5$ and $6$ lying outside of it, as in Figure \ref{fig:appendix} (Right). Given that $\rank \{\gamma_2,\gamma_4,\gamma_7,\gamma_8\}\leq 2$, it follows that $\gamma \in V_{\CCC(A)}$.
After the identification of the double point $\{2,3,9\}$ as a simple point in $A$, the resulting matroid is nilpotent and all of its points have degree at most two. By Theorem~\ref{nil coincide} (i), we conclude that $V_{\CCC(A)} = V_A$. We will end by showing that $V_A \subseteq V_{N}$, which would imply $\gamma \in V_{A}\subset V_{N}$, completing the proof. We will use a similar technique as in \cite{liwski2025minimal}.

Following the constructive approach outlined in the proof of \textup{\cite[Theorem~4.5]{Connectednessandcombinatorialinterplayinthemodulispaceoflinearrangements}}, we find that any element $\gamma \in \Gamma_N$ can be written as \begin{equation}\label{matrix gama}\gamma = \begin{pmatrix} 
1 & 0 & 0 & 1 & 1+v & 1& 1+v & 1 \\
0 & 1 & 0 & 1 & 1 & 0 & 1+z & w+zw \\
0 & 0 & 1 & 1 & 1 & w & 1 & w 
\end{pmatrix},\end{equation}
where the minors corresponding to the bases are nonzero, and the columns are indexed, from left to right, by $\{9,4,5,6,2,3,7,8\}$.
To prove that $V_{A}\subset V_{N}$, it suffices to show that any $\xi\in \Gamma_{A}$ lies in $V_{N}$. We will show this by seeing that we can infinitesimally perturb $\xi$ to obtain an element in $\Gamma_{N}$. Following the construction in the proof of \textup{\cite[Theorem~4.5]{Connectednessandcombinatorialinterplayinthemodulispaceoflinearrangements}}, we find that any element $\xi \in \Gamma_{A}$ has the form $$\xi = \begin{pmatrix}
1 & 0 & 0 & 1 & 1 & 1 & 1 & 1 \\
0 & 1 & 0 & 1 & 0 & 0 & x & y \\
0 & 0 & 1 & 1 & 0 & 0 & 0 & 0
\end{pmatrix},$$
where, the columns correspond to $\{9,4,5,6,2,3,7,8\}$ and the minors corresponding with the bases are non-zero. 
For any $\epsilon$, define $\gamma_\epsilon\in \Gamma_N$ by substituting into the parametrization~\eqref{matrix gama} the values
\[1+v = \frac{1}{\epsilon}, \quad
    w  =  \frac{\epsilon}{x}y, \quad   
    1+z =  \frac{x}{\epsilon}.\]
    
As $\epsilon \to 0$, we get that $\gamma_\epsilon \to\xi$, which implies $\xi \in V_N$. Since $\xi$ was arbitrary, it follows that $\Gamma_{A} \in V_N$, and therefore $V_{A} \in V_N$, completing the proof.
\end{proof}
We now present the irreducible decomposition of the circuit variety of the Pascal configuration.
\begin{theorem} \label{proposition: decomposition pascal configuration}
The circuit variety $V_{\mathcal{C}(M)}$ admits the following irreducible decomposition:
\begin{equation}\label{ecuacion}
V_{\mathcal{C}(M)} = V_M \cup V_{U_{2,9}} \cup_{i=7}^9 V_{M(i)}.
\end{equation}
\end{theorem}

\begin{proof}
The inclusion $\supseteq$ is clear, since the matroids on the right-hand side contain more dependencies than $M$. To establish the reverse inclusion, we must show that any $\gamma\in V_{\mathcal{C}(M)}$ lies in the variety on the right-hand side of~\eqref{ecuacion}. 
First, suppose that there exists $i \in \{7,8,9\}$ such that $\gamma_i = 0$. Without loss of generality, assume $i=7$. Since the matroid $M\backslash \{7\}$, depicted in Figure~\ref{fig:pascal} (d), is nilpotent without points of degree at least three, it follows that $V_{\mathcal{C}(M(7))} = V_{M(7)}$. Consequently, we have $\gamma \in V_{M(7)}$, as desired. Hence, from this point onward, we can assume that $\gamma_{7},\gamma_{8},\gamma_{9}\neq 0$.

\medskip
{\bf Case~1.}~Suppose there exists $p\in [6]$ such that $\gamma_{l_1} \wedge \gamma_{l_2} \neq 0$, where $\{l_1, l_2\} =\mathcal{L}_p$. Without loss of generality, assume $p=1$.
Since $M\backslash \{1\}$ is the matroid from Lemma~\ref{props}, applying this result yields $\restr{\gamma}{\{2,\ldots,9\}}\in V_{M\backslash \{1\}}$. Consequently, we can perturb $\{\gamma_{2},\ldots,\gamma_{9}\}$ to obtain a collection $\{\widetilde{\gamma}_{2},\ldots,\widetilde{\gamma}_{9}\}$ in $\Gamma_{M\backslash \{1\}}$. Extending $\widetilde{\gamma}$ by defining 
$\widetilde{\gamma}_1 = \widetilde{\gamma}_5 \widetilde{\gamma}_7 \wedge \widetilde{\gamma}_6 \widetilde{\gamma}_8$,
we obtain $\widetilde{\gamma}\in V_{M}$, which implies~$\gamma\in V_{M}$.

\medskip
{\bf Case~2.} Suppose that for all points $p \in [6]$, it holds that $\gamma_{l_1} \wedge \gamma_{l_2} = 0$, where $\{l_1, l_2\} =\mathcal{L}_p$. 

\medskip
{\bf Case~2.1.} Suppose there exists $p \in [6]$ such that $\gamma_{p} = 0$. Since this point lies on exactly two lines in $M$, we can redefine $\gamma_{p}$ as a nonzero vector in the intersection $\gamma_{l_{1}} \cap \gamma_{l_{2}}$ of the two lines it belongs to, where $\mathcal{L}_{p} = \{l_{1}, l_{2}\}$. Thus, we may henceforth assume that $\gamma$ contains no loops.

\medskip
{\bf Case~2.2.} Suppose there exists $i\in [6]$ and $j\in [9]$ with $i\neq j$ such that $\gamma_{i}=\gamma_{j}$. Without loss of generality, assume $i=1$. 
First suppose that $\rank \{\gamma_{1},\gamma_{6},\gamma_{8}\}=2$ or $\rank\{\gamma_{1},\gamma_{5},\gamma_{7}\}=2$. By assumption, this implies that the corresponding two lines coincide. Consequently, we can perturb $\gamma_{1}$ to any other vector $\widetilde{\gamma}_{1}$ on this line, ensuring that $\widetilde{\gamma}_{1}\neq \gamma_{j}$, while still remaining in $V_{\mathcal{C}(M)}$.
Now suppose that $\rank \{\gamma_{1},\gamma_{6},\gamma_{8}\}=1$ and $\rank\{\gamma_{1},\gamma_{5},\gamma_{7}\}=1$. In this case, we can perturb $\gamma_{1}$ to any other vector $\widetilde{\gamma}_{1}$, ensuring that $\widetilde{\gamma}_{1}\neq \gamma_{j}$, while still remaining in $V_{\mathcal{C}(M)}$. Thus, we may assume henceforth that there do not exist $i\in [6]$ and $j\in [9]$ with $i\neq j$ such that $\gamma_{i}=\gamma_{j}$.

\medskip
{\bf Case~2.3.} Suppose that Cases~2.1 and~2.2 do not apply. Observe that for every line $l$ of $M$ distinct from $\{7,8,9\}$, the set $\{\gamma_{p}:p\in l\}$ contains three different points and no loops, implying that $\rank(\gamma_{l})=2$. By our assumption in Case~2, these six lines must coincide, that is,
\[\gamma_{\{1,6,8\}} = \gamma_{\{1,5,7\}} = \gamma_{\{2,6,9\}} = \gamma_{\{2,7,4\}} = \gamma_{\{3,4,8\}}=\gamma_{\{3,5,9\}}.\]
It follows that the vectors $\{\gamma_{1},\ldots,\gamma_{9}\}$ all lie in a common two-dimensional subspace. Consequently, we conclude that $\gamma\in V_{\mathcal{C}(U_{2,9})}=V_{U_{2,9}}$.

\medskip

The irreducibility of the varieties in~\eqref{ecuacion} follows from Theorem~\ref{nil coincide}(ii). Moreover, one can easily prove that there are no redundant components in this decomposition. This establishes that the decomposition is indeed the irredundant irreducible decomposition, thereby completing the proof.
    \end{proof}

\begin{figure}
    \centering
    \includegraphics[width=0.9\linewidth]{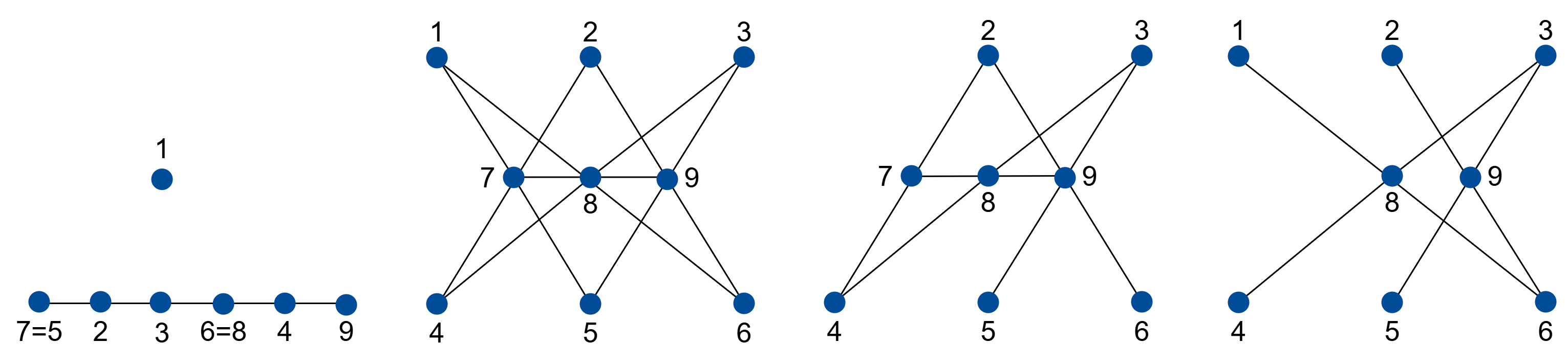}
    \caption{Form left to right: (a) $\pi_M^1$; (b) Pascal configuration ; (c) Pascal configuration without the point 1 ;(d) Pascal configuration without the point 7.}
    \label{fig:pascal}
\end{figure}

\subsubsection{Matroid ideal of Pascal configuration}

We will now provide a generating set for the matroid ideal of the Pascal configuration, up to radical.

\begin{theorem}\label{generators pascal}
    Let $M$ be the matroid associated to the Pascal configuration. Then $$I_{M}=\sqrt{I_{\mathcal{C}(M)}+G_{M}+I_{M}^{\textup{lift}}}.$$
\end{theorem}

\begin{proof}
To prove the equality of the statement, we will show the following equivalent equality of varieties
\begin{equation}\label{lifti}
V_M = V_{\mathcal{C}(M)} \cap V(G_M) \cap V_{M}^{\textup{lift}}.
\end{equation}

The inclusion $\subset$ in~\eqref{lifti} follows directly from Theorem \ref{thm: gamma i V_M^lift => liftable} and Proposition \ref{inc G}. To establish the reverse inclusion, let  $\gamma \in V(G_M) \cap V(I_{M}^{\textup{lift}}) \cap \VCM$. 
Since $\gamma\in V_{\mathcal{C}(M)}$,  by Theorem~\ref{proposition: decomposition pascal configuration} we have
\begin{equation}\label{gamma in}
\gamma\in V_M \cup V_{U_{2,9}} \cup_{i=7}^9 V_{M(i)}.
\end{equation}

\medskip
{\bf Case~1.} Suppose that $\rk(\gamma)=3$.

\medskip
{\bf Case~1.1.} Suppose that $\gamma_{7},\gamma_{8},\gamma_{9}\neq 0$.  Since $\rank(\gamma)=3$, we deduce that $\gamma\not \in V_{U_{2,9}}$. Additionally, because $\gamma_{7},\gamma_{8},\gamma_{9}\neq 0$, we have $\gamma\not \in \cup_{i=7}^9 V_{M(i)}$. Thus, from \eqref{gamma in}, it follows that $\gamma\in V_{M}$, as desired. 

\medskip
{\bf Case~1.2} Suppose there exists a loop $\gamma_i$ for some $i \in \{7,8,9\}$. If there are at least two distinct lines $l \in \mathcal{L}_i$ with $\rk(\gamma_l) = 2$, then, since $\gamma \in V(G_M)$, Proposition~\ref{inc G} ensures that these lines intersect nontrivially. In this case, we redefine $\gamma_i$ as the intersection of these lines. 
After applying this procedure to all such loops, we may assume that each remaining loop has at most one line $l \in \mathcal{L}_i$ with $\rk(l) = 2$.

Now perturb $\gamma_i$ to a nonzero vector in $\gamma_{l}$ for some $l \in \mathcal{L}_i$ with $\rk(l) = 2$. If no such line exists, we redefine $\gamma_i$ arbitrarily. The resulting configuration remains in $\VCM$, since any line passing through $\gamma_i$ with rank one will still have rank at most two after this perturbation. Repeating this process for all points in $\{7,8,9\}$, we redefine $\gamma$ such that $\gamma_{7},\gamma_{8},\gamma_{9}\neq 0$ while ensuring that the configuration remains in $\VCM$. At this stage, we fall into Case~1.1, from which it follows that $\gamma\in V_{M}$, as desired. 

\medskip
{\bf Case~2.} Suppose that $\rk(\gamma) =2$. In this case, all vectors of $\gamma$ lie on a single line in the projective plane, which we denote by $l$. 

\medskip
{\bf Case~2.1.} Suppose that $\gamma_{7}, \gamma_8, \gamma_9 \neq 0$. Since $\gamma \in V(I_{M}^{\textup{lift}})$, we can use Theorem \ref{thm: gamma i V_M^lift => liftable} to lift $\gamma$ to $\widetilde{\gamma} \in \VCM$, such that $\rk(\widetilde{\gamma}) = 3$ and $ \widetilde{\gamma}_7, \widetilde{\gamma}_8, \widetilde{\gamma}_9 \neq 0$. By Case~1.1, this ensures that $\widetilde{\gamma} \in V_M$, and consequently $\gamma\in V_{M}$, as desired.

\medskip
{\bf Case~2.2.} Suppose that at least one of $\gamma_7, \gamma_8, \gamma_9$ is a loop. If multiple elements in $\{\gamma_{7},\gamma_{8},\gamma_{9}\}$ are loops, redefine all  but one of them as arbitrary points on $l$. Consequently, we may assume without loss of generality that the only loop in $\{\gamma_{7},\gamma_{8},\gamma_{9}\}$ is $\gamma_{7}$. 
Let $q\in \mathbb{P}^{2}$ be any point outside $l$. By Theorem~\ref{qgamma is constant}, we obtain
$\dim_q^\gamma(M \backslash \{7\}) = \dim_q^\gamma(S_M) = 4$. 
Applying Proposition~\ref{prop: nilpotent add point liftable}, we can redefine $\gamma_7$ as a point on $l$ such that the resulting configuration is liftable. Thus, we can perturb $\gamma$ to $\widetilde{\gamma} \in V_M$.

\medskip
{\bf Case~3.} Suppose that $\rk(\gamma) = 1$. In this case, all points either coincide at a fixed point or are loops. If at least one point is a loop, perturb the remaining points to reduce to Case~2.2. If there are no loops, then all points coincide. We will use the same technique as in \cite{liwski2025minimal}. By applying a projective transformation, we may assume that this common point is $(1,0,0)$.

Following the constructive approach outlined in the proof of \textup{\cite[Theorem~4.5]{Connectednessandcombinatorialinterplayinthemodulispaceoflinearrangements}}, we find that any element of $\Gamma_M$ can be expressed as $$\begin{pmatrix}
           1 & 1 & 1 & 1 & 1+x & 1+y & 1+y+z & -(1+z)(x+1)+1+y+z & 2x+1 \\
           0 & \epsilon & 0 & \epsilon & \epsilon & \epsilon & \epsilon + \epsilon z & 0 & (1+z) \epsilon \\
           0 & 0 & \eta & \eta & 0 & \eta & \eta & \eta & \eta x
       \end{pmatrix}.$$
       Letting $x,y,\eta, \epsilon, z \to 0$, we see that the configuration in which all points coincide can be realized as a limit of configurations in $\Gamma_M$. It follows that $\gamma \in V_M$.
\end{proof}

Although Theorem~\ref{generators pascal} establishes that the ideals $I_{\CCC(M)}, G_M$ and $I_M^{\textup{lift}}$ generate $I_{M}$ up to radical, it does not present an explicit generating set. The following remark addresses this point.
\begin{remark} \label{remark pascal}
By Theorem~\ref{generators pascal}, to find explicit generators for $I_M$, up to radical, it suffices to find explicit generators for $I_{\CCC(M)}, G_M$ and $I_M^{\textup{lift}}$.
\begin{itemize}[noitemsep]
\item The circuit polynomials are precisely the bracket polynomials corresponding to the circuits of $M$, see Remark~\ref{generating set ICM}. In this case, they are:
$\bigl\{[168],[157],[247],[269],[348],[359],[789]\bigr\}.$
\item Regarding the Grassmann–Cayley polynomials in $G_M$, an inspection of the proof of Theorem~\ref{generators pascal} gives that it suffices to consider the following polynomials:
\begin{itemize}
\item[{\rm (i)}] 
$[153][142][546][326]-[154][132][536][426]$
arising from $(15 \wedge 24) \vee (16 \wedge 34) \vee (35 \wedge 26)=0$.
\item[{\rm (ii)}] 
$[526][361][734]-[326][361][754]+[326][461][753]$
corresponding to $7\vee (53\wedge 26)\vee (34\wedge 61)=0$,
along with the two analogous polynomials in $G_M$ obtained from the expressions
\[8\vee (51\wedge 24)\vee (35\wedge 62)=0,\qquad \text{and} \qquad 9\vee (43\wedge 16)\vee (24\wedge 51)=0.\]
\item[{\rm (iii)}] 
$[749][361]-[461][739]\in G_M$
corresponding to the expression
$7\vee 9 \vee (34\wedge 61)=0$,
together with its two analogues derived from
\[7\vee 8 \vee (35\wedge 62)=0,\qquad \text{and} \qquad 9\vee 8 \vee (15\wedge 42)=0.\]
\end{itemize}
These are precisely the seven Grassmann–Cayley polynomials identified in \textup{\cite[Theorem~3.0.2]{Sidman}}. We refer the reader to that work for the details.
\item Looking into the proof of Theorem~\ref{generators pascal}, we note that we only use the assumption $\gamma \in V(I_M^{\textup{lift}})$ to ensure the existence of a non-degenerate lifting of $\{\gamma_1,\ldots,\gamma_9\}$ to a collection in $V_{\CCC(M)}$. Thus, by~\cite[Lemma 3.6]{liwski2024pavingmatroidsdefiningequations}, it suffices to consider the $7\times 7$ minors of the liftability matrices $\mathcal{M}_{q}(M)$.

Although the set of all such minors is not finite, since $q$ ranges over $\CC^{3}$, the discussion in \cite[\S6]{liwski2024pavingmatroidsdefiningequations} shows that one obtains a finite generating set by selecting vectors $q_1, \ldots q_9\in \CC^{3}$ and replacing the vector $q$ appearing in the brackets of column $i$ with $q_i$. This also gives us a polynomial inside $I_M$. By multilinearity of the determinant, it is enough to take each $q_i$ from the canonical basis $\{e_1, e_2, e_3\}$ of $\CC^3$. The resulting polynomials are the $7\times 7$ minors of the matrices 

 $$\begin{pmatrix}
       [68q_1] & 0 & 0 & 0 & 0 & -[18q_6] & 0 & [16q_8] & 0 \\
       [57q_1] & 0 & 0 & 0 & -[17q_5] & 0 & [15q_7] & 0 & 0 \\
       0 & [47q_2] & 0 & -[27q_4] & 0 & 0 & [24q_7] & 0 & 0 \\
       0 & [69q_2] & 0 & 0 & 0 & -[29q_6] & 0 & 0 & [26q_9] \\
       0 & 0 & [48q_3] & -[38q_4] & 0 & 0 & 0 & [34q_8] & 0 \\
       0 & 0 & [59q_3] & 0 & -[39q_5] & 0 & 0 & 0& [35q_9] \\
       0 & 0 & 0 & 0 & 0 & 0 & [89q_7] & -[79q_8] & [78q_9] \\
   \end{pmatrix},$$
where $q_1, \ldots, q_9\in \{e_1, e_2, e_3\}$. The total number of such polynomials is $\textstyle \binom{9}{7}3^{9}=708{,}588$.
\end{itemize}
\end{remark}

\subsection{Pappus configuration}
We now study the Pappus configuration, the simple rank-three matroid in Figure~\ref{fig:Pappus, Pappus, I_i, J_i} (Left). Specifically, we find a finite generating set for its matroid ideal.
In this subsection, we denote this matroid by $M$.

\medskip

We first state the following result from \cite[\S5.4]{liwski2025minimal}.

\begin{theorem}\label{deco pappus}
   The circuit variety of $M$ admits the following irreducible decomposition
\begin{equation}\label{decomposition of pappus} \VCM = V_{M} \cup V_{U_{2,9}} \cup_{i=1}^{18} V_{I_i} \cup_{i=1}^3 V_{J_i} \cup_{i=1}^9 V_{\pi_M^i}, \end{equation}
where the matroids in the decomposition are the following:

\begin{itemize}
\item $U_{2,9}$ denotes the uniform matroid of rank two on the ground set $[9]$, see Definition~\ref{uniform 3}.
\item We denote by $I_i$ for $i\in [18]$ the matroid obtained from $M$ by making one of its
points a loop and adding one of the circuits $\{1, 4, 9\}, \{2,5,8\}$ or $\{3,6,7\}$, see Figure~\ref{fig:Pappus, Pappus, I_i, J_i} (Center). 
\item We denote $J_{1},J_{2},J_{3}$ for the matroids derived from $M$ by making loops all three points of
one of the triples $\{1,4,9\}, \{3,6,7\}$ or $ \{2,5,8\}$, see Figure~\ref{fig:Pappus, Pappus, I_i, J_i} (Right).
\item The matroids $\pi_M^i$ are defined in Notation~\ref{notation pi}. See Figure~\ref{fig:Pappus, K_i, L__i, N} (Left) for an illustration.
\end{itemize}
\end{theorem}

To determine the defining equations of $M$, we use the following proposition, which decomposes the circuit variety of the matroids $M(i)$ obtained from $M$ by making one point a loop. Since all cases are analogous, we focus on $i=9$ in the proposition below.  

\begin{proposition} \label{prop: Pappus M(9)}
The circuit variety of  $M(9)$ from Figure \ref{fig:Pappus, K_i, L__i, N} (Center) satisfies:
    \begin{equation} \label{eq: VCM(9) final decomposition}
V_{\CCC(M(9))} = V_{M(9)} \cup V_{U_{2,9}(9)} \cup_{j=1,4} V_{M(9,j)} \cup V_{M(1,4,9)}.
\end{equation}
\end{proposition}
\begin{proof}
The decomposition of 
$V_{\CCC(M(9))}$ corresponds to that of $V_{\CCC(M \backslash \{9\})}$, so it suffices to decompose $V_{\CCC(M \backslash \{9\})}$. For simplicity, denote $N := M \backslash \{9\}$. With this notation, it is enough to prove that:
\begin{equation}\label{ccn}
V_{\CCC(N)} = V_{N} \cup V_{U_{2,8}} \cup \bigcup_{j \in \{1,4\}} V_{N(j)} \cup V_{N(1,4)}.
\end{equation}
The inclusion $\supseteq$ in~\eqref{ccn} is immediate since each matroid variety on the right-hand side corresponds to matroids with strictly more dependencies than $N$. To prove the reverse inclusion, we will show that any collection $\gamma \in V_{\CCC(N)}$ must lie in one of the varieties appearing on the right-hand side of~\eqref{ccn}.

Suppose first that either $\gamma_1 = 0$ or $\gamma_4 = 0$, and without loss of generality, assume that $\gamma_1 = 0$. The matroid $M(1,9)$ is illustrated in Figure~\ref{fig:Pappus, K_i, L__i, N} (Right). By Lemma~5.5 (iv) of \cite{liwski2025minimal}, we have
$V_{\mathcal{C}(N(1))} = V_{N(1)} \cup V_{N(1,4)}$,
which implies
$\gamma \in V_{N(1)} \cup V_{N(1,4)}$.

Now assume that $\gamma_1,\gamma_4 \neq 0$. Then, by Lemma~\ref{lemma: third config loop matroid variety}, we may assume that $\gamma$ has no loops.

\medskip

\begin{mycases}
    \case{There exists $p \in [8] \backslash \{1,4\}$ such that $\gamma_{l_1} \wedge \gamma_{l_2} \neq 0$, where $\{l_1, l_2\} =\mathcal{L}_p$.}

\smallskip

In this case, since $N \backslash \{p\}$ is nilpotent and all its points have degree at most two, Theorem~\ref{nil coincide} (i) allows us to perturb the collection $\{\gamma_i : i \in [8] \backslash \{p\}\}$ to obtain vectors $\{\widetilde{\gamma}_i : i \in [8] \backslash \{p\}\} \in \Gamma|_{N \backslash \{p\}}$. We then extend $\widetilde{\gamma}$ by defining $\widetilde{\gamma}_p$ as $\widetilde{\gamma}_{l_1} \wedge \widetilde{\gamma}_{l_2}$, where $\mathcal{L}_p = \{l_1, l_2\}$, thereby obtaining a collection $\widetilde{\gamma} \in V_{N}$. Since $\widetilde{\gamma}$ is a perturbation of $\gamma$, it follows that $\gamma \in V_{N}$, as desired.

\medskip

\case{For all $p \in [8]\backslash \{1,4\}$ it holds that $\gamma_{l_1} \wedge \gamma_{l_2} = 0$, where $\{l_1, l_2\} =\mathcal{L}_p$.} 

\medskip

{\bf Case~2.1.} Suppose there exist distinct $i,j \in [8]$ such that either $i,j \in [8] \setminus \{1,4\}$ or $i \in [8] \setminus \{1,4\}$ and $j \in \{1,4\}$, for which $\gamma_i = \gamma_j$. Assume without loss of generality that $i=1$ and $j=2$. Since $\gamma_4 \gamma_7 \wedge \gamma_1 \gamma_3 = 0$, the pairs $\{\gamma_4, \gamma_7\}$ and $\{\gamma_1, \gamma_3\}$ cannot span distinct lines in $\mathbb{P}^2$. If both pairs span the same line, then $\gamma_2$ can be perturbed to any point on that line away from $\gamma_1$, yielding a collection that remains in $V_{\CCC(N)}$. If both pairs span a one-dimensional space, then $\gamma_2$ can be perturbed to an arbitrary point away from $\gamma_1$, producing a collection still in $V_{\CCC(N)}$. Repeating this argument as needed, we may assume without loss of generality that the collection in $V_{\CCC(N)}$ does not present this type of double points. We now analyze this case.

\medskip

{\bf Case~2.2.} Suppose that $\gamma_i \neq \gamma_j$ for all distinct $i,j \in [8] \setminus \{1,4\}$, as well as for all $i \in [8] \setminus \{1,4\}$ and $j \in \{1,4\}$. Under the assumption of Case~2, we have the following equalities:
\[
\gamma_{\{1,2,3\}} = \gamma_{\{2,7,4\}} = \gamma_{\{3,4,8\}} = \gamma_{\{1,6,8\}} = \gamma_{\{4,5,6\}} = \gamma_{\{1,5,7\}},
\]
which imply that the vectors $\{\gamma_1, \ldots, \gamma_8\}$ are all collinear in $\mathbb{P}^2$. Hence, $\gamma \in V_{U_{2,8}}$, as desired. 
\end{mycases}
\end{proof}

We now present the main result of this subsection, a generating set for the matroid ideal $I
_M$.

\begin{theorem}\label{thm:Pappus}
    Let $M$ be the matroid associated to the Pappus configuration. Then $$I_{M}=\sqrt{I_{\mathcal{C}(M)}+G_{M}+I_{M}^{\textup{lift}}}.$$
\end{theorem}

\begin{proof}
To prove the equality of the statement, we will show the following equivalent equality of varieties
\begin{equation}\label{lifting}
V_M = V_{\mathcal{C}(M)} \cap V(G_M) \cap V_{M}^{\textup{lift}}.
\end{equation}
The inclusion $\subset$ in~\eqref{lifting} follows directly from Theorem~\ref{thm: gamma i V_M^lift => liftable} and Proposition~\ref{inc G}. 
To prove the reverse inclusion, let $\gamma\in V(G_{M})\cap V_{M}^{\text{lift}}\cap V_{\mathcal{C}(M)}$. To show that $\gamma\in V_{M}$, we consider the following cases:  

\begin{figure}
    \centering
    \includegraphics[width=0.6\linewidth]{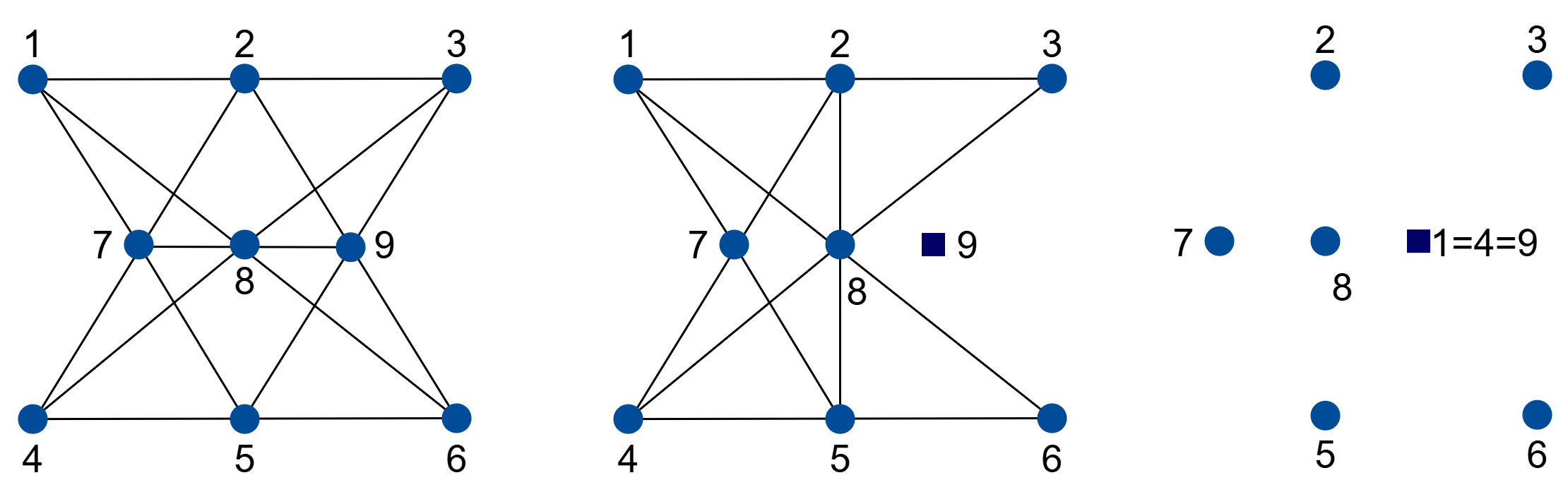}
    \caption{(Left) Pappus Configuration, (Center) $I_9$, (Right) $J_1$}
    \label{fig:Pappus, Pappus, I_i, J_i}
\end{figure}
\begin{figure}
    \centering
    \includegraphics[width=0.6\linewidth]{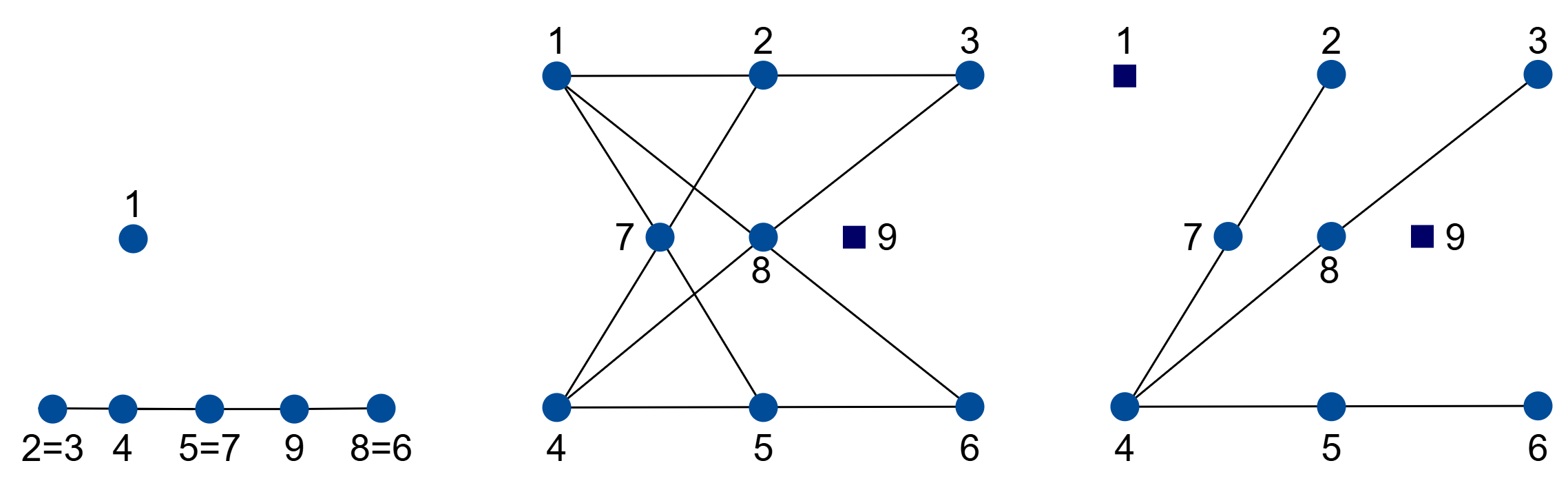}
    \caption{(Left) $\pi_M^i$; (Center) $M(9)$; (Right) $M(1,9)$.}
    \label{fig:Pappus, K_i, L__i, N}
\end{figure}

\medskip

\begin{mycases}
\case {\bf $\gamma$ has no loops.}

\medskip

In this case, Theorem~\ref{deco pappus} shows that  $\gamma \in V_{U_{2,9}}$ or $\gamma \in V_{\pi_M^i}$ for some $i\in [9]$ or $\gamma \in V_{M}$. Since we must prove that $\gamma \notin V_M$, it remains to consider the first two possibilities.

\medskip

\textbf{Case~1.1.} Suppose that $\gamma \in V_{\pi_M^i}$ for some $i \in [9]$. Then, all vectors of $\gamma$, except one, lie on a line in $\mathbb{P}^2$. Without loss of generality, assume this exceptional vector corresponds to the point $1$. Since $M \backslash \{1\}$ is a full-rank submatroid of $M$, Theorem~\ref{thm: gamma i V_M^lift => liftable} allows us to infinitesimally lift the vectors $\{\gamma_2, \ldots, \gamma_9\}$ from $\gamma_1$, to obtain vectors
$\{\widetilde{\gamma}_2, \ldots, \widetilde{\gamma}_9\} \in V_{\CCC(M \backslash \{1\})}$
with $\operatorname{rank}\{\gamma_2, \ldots, \gamma_9\} = 3$. Letting $\widetilde{\gamma}$ denote the resulting modified collection, we have $\widetilde{\gamma} \in \VCM$. Since $\widetilde{\gamma}$ has no loops and every subset of eight vectors has rank three, it follows from Theorem~\ref{deco pappus} that $\widetilde{\gamma} \in V_M$.

\medskip

{\bf Case~1.2}. Suppose that $\gamma\in V_{U_{2,9}}$. Since $\gamma \in \VMLIFT$, we can infinitesimally lift the configuration from an arbitrary point to obtain a non-degenerate configuration $\widetilde{\gamma} \in \VCM$. 
        If $\widetilde{\gamma} \in V_{M}$, the claim follows.
Otherwise, $\widetilde{\gamma} \in V_{\pi_M^i}$ for some $i\in [9]$. Without loss of generality, assume $i=1$. Since arbitrarily small perturbations do not create new dependencies, this forces the original configuration to satisfy $\gamma_{2}=\gamma_3,\gamma_5=\gamma_7,\gamma_6=\gamma_8$. By perturbing $\gamma_1$ slightly so that it no longer lies on the line spanned by $\{\gamma_2, \cdots, \gamma_9\}$, we reduce to the setting of Case~1.1, which implies that $\gamma\in V_{M}$.

\medskip

\case\textbf{$\gamma$ has exactly one loop.} 

\medskip 
Without loss of generality, assume that the loop corresponds to $9$. By Lemma 5.5 (iii) of \cite{liwski2025minimal}, we have $V_{M(9)} \subseteq V_M$, so it suffices to prove that $\gamma \in V_{M(9)}$. The matroid $M(9)$ is shown in Figure~\ref{fig:Pappus, K_i, L__i, N} (Center).
According to Equation~\eqref{eq: VCM(9) final decomposition}, either $\gamma \in V_{M(9)}$ or $\gamma \in V_{U_{2,9}(9)}$. If the former holds, then $\gamma \in V_M$, as desired. Otherwise, assume $\gamma \in V_{U_{2,9}(9)}$. Since $\gamma \in \VMLIFT$, we can infinitesimally lift the vectors $\{\gamma_1, \ldots, \gamma_8\}$ from an arbitrary vector $q$ outside their span, yielding a non-degenerate collection $\{\widetilde{\gamma}_1, \ldots, \widetilde{\gamma}_8\}$ of rank three in $V_{\CCC(M \backslash \{9\})}$. Denote by $\widetilde{\gamma}$ the perturbed collection of vectors, where $\widetilde{\gamma}_9 = 0$. Then $\widetilde{\gamma}$ does not have any loops apart from $\widetilde{\gamma}_9$ and is non-degenerate, so it must be in $V_{M(9)}$. Since $V_{M(9)} \subseteq V_M$ and $\widetilde{\gamma}$ is a perturbation of $\gamma$, we conclude that $\gamma \in V_M$.

\medskip

\case\textbf{$\gamma$ has exactly two loops.}

\medskip

{\bf Case~3.1.} Suppose the two loops correspond to elements in $[9]$ that do not lie on a common line in $M$. Without loss of generality, assume these loops correspond to $1$ and $9$. The matroid $M(1,9)$ is illustrated in Figure~\ref{fig:Pappus, K_i, L__i, N} (Right).

 \medskip
 
{\bf Case~3.1.1.} Assume that the vectors $\gamma_{2}, \ldots, \gamma_{8}$ are collinear. By applying a small perturbation within the line they span, we may further assume that they are pairwise distinct. By Theorem~\ref{qgamma is constant}, 
$\dim_q^\gamma(M \backslash \{1, 9\}) = 1 + 3 = 4$,
where $q \in \mathbb{C}^3$ is any vector in general position with respect to $\gamma$. Moreover, $M \backslash \{1, 9\}$ is nilpotent, the point $1$ has degree three in $M$, and all the $\gamma_i$ are pairwise distinct and nonzero for $i \in \{2, \ldots, 8\}$.
We may then apply Proposition~\ref{prop: nilpotent add point liftable} to redefine $\gamma_1$ along the line spanned by $\gamma_2, \ldots, \gamma_8$, so that $\gamma|_{[8]} \in V_{M \backslash \{9\}}^{\mathrm{lift}}$. This reduces to Case~2, from which we conclude that $\gamma \in V_M$.

\medskip
{\bf Case~3.1.2.} Suppose now that $\operatorname{rank}\{\gamma_{2},\ldots,\gamma_{8}\} = 3$.  
One can redefine $\gamma_{1}$ from zero to a point lying on all the lines $\gamma_l \subset \mathbb{P}^{2}$ for which $l$ is a line of $M$ containing $1$. Since $\gamma \in V(G_M)$, Proposition~\ref{inc G} ensures that these lines intersect nontrivially in $\mathbb{P}^{2}$. Then the collection of vectors remains within $V_{\CCC(M(9))}$. 
Let $\widetilde{\gamma}$ denote this modified collection. Given that $\widetilde{\gamma} \in V_{\CCC(M(9))}$, has rank three, and contains only one loop, it follows from Equation~\eqref{eq: VCM(9) final decomposition} that $\widetilde{\gamma} \in V_{M(9)}$. Since $V_{M(9)} \subset V_{M}$, we deduce that $\widetilde{\gamma} \in V_{M}$. Because $\widetilde{\gamma}$ is a perturbation of $\gamma$, it follows that $\gamma \in V_{M}$.

\medskip

{\bf Case~3.2.} Suppose the two loops correspond to elements of $[9]$ that lie on a common line in $M$. In this case, by Lemma 5.5 (ii) of \cite{liwski2025minimal}, the circuit variety of this matroid is contained within that of $V_{M}$.

\medskip 

\case{\bf There are three loops.}

\medskip

{\bf Case~4.1.} Suppose that two of the three loops correspond to elements of $[9]$ that lie on a common line in $M$. In this case, by Lemma 5.5 (ii) of \cite{liwski2025minimal}, the circuit variety of $M$ is a subset of $V_{M}$.

\medskip

{\bf Case~4.2.} Suppose that no two of the three loops correspond to elements of $[9]$ that lie on a common line in $M$. In this case, we may assume, without loss of generality, that the loops correspond to the elements $1$, $4$, and $9$. Since $\gamma \in V(G_M)$, we can redefine $\gamma_4$ as a nonzero vector while keeping the collection of vectors within $V_{\CCC(M)}$. This reduces us to Case~3, where the argument relies only on the concurrency of the lines $\gamma_{\{1,2,3\}}$, $\gamma_{\{1,6,8\}}$, and $\gamma_{\{1,5,7\}}$, a property that holds after the perturbation.

\medskip

\case{\bf There are at least four loops.}

\medskip

In this case, at least two of these four loops correspond to elements in $[9]$ that belong to a common line in $M$. It then follows from Lemma~5.5(ii) of \cite{liwski2025minimal} that the circuit variety of $M$ is a subset of $V_{M}$.
\end{mycases}

\medskip

We have established that $\gamma\in V_{M}$ in each of the possible cases, thereby completing the proof.
\end{proof}

Although Theorem~\ref{thm:Pappus} establishes that the ideals $I_{\CCC(M)}$, $G_M$, and $I_M^{\textup{lift}}$ generate $I_{M}$ up to radical, it does not provide an explicit generating set. The following remark addresses this point.

\begin{remark} \label{remark pappus}
By Theorem~\ref{generators pascal}, to find explicit generators for $I_M$, up to radical, it suffices to find explicit generators for $I_{\CCC(M)}, G_M$ and $I_M^{\textup{lift}}$.
\begin{itemize}[noitemsep]
 \item The circuit polynomials are: 
     $\bigl\{[123],[168],[157],[247],[269],[789],[456],[348],[359]\bigr\}.$
\item Concerning the Grassmann–Cayley polynomials, by the proof of Theorem~\ref{thm:Pappus}, it suffices to consider the nine polynomials arising from the nine triples of concurrent lines in $M$ given by:
\begin{equation*}
\begin{aligned}\bigl\{&[235][768]-[237][568], [134][769]-[137][469], [124][859]-[128][459],\\
&[241][589]-[245][189], 
     [791][634]-[796][134],[263][578]-[265][378],\\
&[273][856]-[278][356],[461][739]-[467][139],[291][845]-[298][145]\bigr\}.
\end{aligned}
\end{equation*}
\item From the proof of Theorem~\ref{thm:Pappus}, we observe that the only instances where the assumption $\gamma \in V(I_M^{\textup{lift}})$ is used are: (1) to guarantee the existence of a non-degenerate lifting of the vectors $\{\gamma_{1}, \ldots, \gamma_{9}\}$ to a collection in $V_{\CCC(M)}$, and (2) to ensure that for each $i \in [9]$, the eight vectors $\{\gamma_1, \ldots, \gamma_{i-1}, \hat{\gamma}_i, \gamma_{i+1}, \ldots, \gamma_9\}$ have a non-degenerate lifting to a collection in $V_{\CCC(M \backslash \{i\})}$. Thus, by \textup{\cite[Lemma~3.6]{liwski2024pavingmatroidsdefiningequations}}, it suffices to consider the $7 \times 7$ minors of the liftability matrices $\mathcal{M}_q(M)$, together with the $6 \times 6$ minors of the liftability matrices $\mathcal{M}_q(M \backslash \{i\})$ for $i \in [9]$.

Although the set of all such minors is not finite, since $q$ ranges over $\CC^{3}$, we may adopt the same strategy used in Remark~\ref{remark pascal} to obtain a finite generating set. Namely, we select vectors $q_{1},\ldots,q_{9}\in \{e_{1},e_{2},e_{3}\}$ and replace the vector $q$ appearing in the bracket of column $i$ with $q_i$. The resulting polynomials are the $7\times 7$ minors of the matrices

$$\begin{pmatrix}
       [23q_1] & [-13q_2] & [12q_3] & 0 & 0 & 0 & 0& 0 & 0 \\
       [68q_1] & 0 & 0 & 0 & 0 & -[18q_6] & 0 & [16q_8] & 0 \\
       [57q_1] & 0 & 0 & 0 & -[17q_5] & 0 & [15q_7] & 0 & 0 \\
       0 & [47q_2] & 0 & -[27q_4] & 0 & 0 & [24q_7] & 0 & 0 \\
       0 & [69q_2] & 0 & 0 & 0 & -[29q_6] & 0 & 0 & [26q_9] \\
       0 & 0 & [48q_3] & -[38q_4] & 0 & 0 & 0 & [34q_8] & 0 \\
       0 & 0 & [59q_3] & 0 & -[39q_5] & 0 & 0 & 0& [35q_9] \\
       0 & 0 & 0 & 0 & 0 & 0 & [89q_7] & -[79q_8] & [78q_9] \\
       0 & 0 & 0 & [56q_4] & -[46q_5] & [45q_6] & 0 & 0 & 0
   \end{pmatrix},$$
where $q_{1},\ldots,q_{9}\in \{e_{1},e_{2},e_{3}\}$, and the $6 \times 6$ minors of the matrices
$$\begin{pmatrix}
       [23q_1] & [-13q_2] & [12q_3] & 0 & 0 & 0 & 0& 0 \\
       [68q_1] & 0 & 0 & 0 & 0 & -[18q_6] & 0 & [16q_8]  \\
       [57q_1] & 0 & 0 & 0 & -[17q_5] & 0 & [15q_7] & 0 \\
       0 & [47q_2] & 0 & -[27q_4] & 0 & 0 & [24q_7] & 0  \\
       0 & 0 & [48q_3] & -[38q_4] & 0 & 0 & 0 & [34q_8]  \\
       0 & 0 & 0 & [56q_4] & -[46q_5] & [45q_6] & 0 & 0 
   \end{pmatrix},$$
   where $q_1, \ldots, q_8 \in \{e_1,e_2,e_3\}$, 
   along with the $6\times 6$ minors of the matrices analogous to this, corresponding to omitting each index $i=1,\ldots,8$. The total number of such polynomials is:
\[\textstyle \binom{9}{7}3^{9}+\binom{8}{6}3^{10}=2{,}361{,}960 .\]
\end{itemize}
\end{remark}

\smallskip

\noindent{\bf Acknowledgement.} 
This work is based on the Master's thesis of the third author. F.M and E.L were partially supported by the FWO grants G0F5921N (Odysseus) and G023721N, and the grant iBOF/23/064 from KU Leuven. E.L. was supported by PhD fellowship 1126125N.


\bibliographystyle{abbrv}
\bibliography{references}

\begin{thebibliography}{10}

\bibitem{bruns2003determinantal}
W.~Bruns and A.~Conca.
\newblock Gr{\"o}bner bases and determinantal ideals.
\newblock In {\em Commutative Algebra, Singularities and Computer Algebra}, pages 9--66. Springer Netherlands, 2003.

\bibitem{caines2022lattice}
P.~Caines, F.~Mohammadi, E.~S{\'a}enz-de Cabez{\'o}n, and H.~Wynn.
\newblock Lattice conditional independence models and {H}ibi ideals.
\newblock {\em Transactions of the London Mathematical Society}, 9(1):1--19, 2022.

\bibitem{clarke2021matroid}
O.~Clarke, K.~Grace, F.~Mohammadi, and H.~Motwani.
\newblock Matroid stratifications of hypergraph varieties, their realization spaces, and discrete conditional independence models.
\newblock {\em International Mathematics Research Notices}, page rnac268, 2022.

\bibitem{clarke2024liftablepointlineconfigurationsdefining}
O.~Clarke, G.~Masiero, and F.~Mohammadi.
\newblock Liftable point-line configurations: Defining equations and irreducibility of associated matroid and circuit varieties.
\newblock {\em Mathematics}, 12(19):3041, 2024.

\bibitem{clarke2022conditional}
O.~Clarke, F.~Mohammadi, and H.~Motwani.
\newblock Conditional probabilities via line arrangements and point configurations.
\newblock {\em Linear and Multilinear Algebra}, 70(20):5268--5300, 2022.

\bibitem{DrtonSturmfelsSullivant09:Algebraic_Statistics}
M.~Drton, B.~Sturmfels, and S.~Sullivant.
\newblock {\em Lectures on Algebraic Statistics}, volume~39.
\newblock Birkh\"{a}user, Basel, first edition, 2009.

\bibitem{ene2013determinantal}
V.~Ene, J.~Herzog, T.~Hibi, and F.~Mohammadi.
\newblock Determinantal facet ideals.
\newblock {\em Michigan Mathematical Journal}, 62(1):39--57, 2013.

\bibitem{feher2012equivariant}
L.~M. Feh{\'e}r, A.~N{\'e}methi, and R.~Rim{\'a}nyi.
\newblock Equivariant classes of matrix matroid varieties.
\newblock {\em Commentarii Mathematici Helvetici}, 87(4):861--889, 2012.

\bibitem{gelfand1987combinatorial}
I.~Gelfand, M.~Goresky, R.~MacPherson, and V.~Serganova.
\newblock Combinatorial geometries, convex polyhedra, and schubert cells.
\newblock {\em Advances in Mathematics}, 63(3):301--316, 1987.

\bibitem{graver1993combinatorial}
J.~E. Graver, B.~Servatius, and H.~Servatius.
\newblock {\em Combinatorial rigidity}.
\newblock Number~2 in Mathematical Sciences Series. American Mathematical Soc., 1993.

\bibitem{Connectednessandcombinatorialinterplayinthemodulispaceoflinearrangements}
B.~Guerville-Ballé and J.~Viu-Sos.
\newblock Connectedness and combinatorial interplay in the moduli space of line arrangements, 2023.

\bibitem{Hartshorne}
R.~Hartshorne.
\newblock {\em Algebraic geometry}, volume~52.
\newblock Springer Science \& Business Media, 2013.

\bibitem{herzog2010binomial}
J.~Herzog, T.~Hibi, F.~Hreinsd{\'o}ttir, T.~Kahle, and J.~Rauh.
\newblock Binomial edge ideals and conditional independence statements.
\newblock {\em Advances in Applied Mathematics}, 3(45):317--333, 2010.

\bibitem{hocsten2004ideals}
S.~Ho{\c{s}}ten and S.~Sullivant.
\newblock Ideals of adjacent minors.
\newblock {\em Journal of Algebra}, 277(2):615--642, 2004.

\bibitem{jackson2024maximal}
B.~Jackson and S.-i. Tanigawa.
\newblock Maximal matroids in weak order posets.
\newblock {\em Journal of Combinatorial Theory, Series B}, 165:20--46, 2024.

\bibitem{knutson2013positroid}
A.~Knutson, T.~Lam, and D.~E. Speyer.
\newblock Positroid varieties: juggling and geometry.
\newblock {\em Compositio Mathematica}, 149(10):1710--1752, 2013.

\bibitem{liwski2024pavingmatroidsdefiningequations}
E.~Liwski and F.~Mohammadi.
\newblock Paving matroids: defining equations and associated varieties.
\newblock {\em arXiv preprint arXiv:2403.13718}, 2024.

\bibitem{liwski2025minimal}
E.~Liwski and F.~Mohammadi.
\newblock Minimal matroids in dependency posets: algorithms and applications to computing irreducible decompositions of circuit varieties.
\newblock {\em arXiv preprint arXiv:2502.00799}, 2025.

\bibitem{liwski2025efficient}
E.~Liwski, F.~Mohammadi, and R.~Pr{\'e}bet.
\newblock Efficient algorithms for minimal matroid extensions and irreducible decompositions of circuit varieties.
\newblock {\em arXiv preprint arXiv:2504.16632}, 2025.

\bibitem{Oxley}
J.~Oxley.
\newblock {\em Matroid Theory}.
\newblock Second edition, Oxford University Press, 2011.

\bibitem{pfister2019primary}
G.~Pfister and A.~Steenpass.
\newblock On the primary decomposition of some determinantal hyperedge ideal.
\newblock {\em Journal of Symbolic Computation}, 103:14--21, 2019.

\bibitem{piff1970vector}
M.~J. Piff and D.~J. Welsh.
\newblock On the vector representation of matroids.
\newblock {\em Journal of the London Mathematical Society}, 2(2):284--288, 1970.

\bibitem{poljak1984amalgamation}
S.~Poljak and D.~Turz{\'\i}k.
\newblock Amalgamation over uniform matroids.
\newblock {\em Czechoslovak Mathematical Journal}, 34(2):239--246, 1984.

\bibitem{Sidman}
J.~Sidman, W.~Traves, and A.~Wheeler.
\newblock Geometric equations for matroid varieties.
\newblock {\em Journal of Combinatorial Theory, Series A}, 178:105360, 2021.

\bibitem{maximum}
M.~Sitharam and A.~Vince.
\newblock The maximum matroid of a graph.
\newblock {\em arXiv preprint arXiv:1910.05390}.

\bibitem{Studeny05:Probabilistic_CI_structures}
M.~Studený.
\newblock {\em Probabilistic conditional independence structures}.
\newblock Springer, London, 2005.

\bibitem{sturmfels1989matroid}
B.~Sturmfels.
\newblock On the matroid stratification of {G}rassmann varieties, specialization of coordinates, and a problem of {N}. {W}hite.
\newblock {\em Advances in Mathematics}, 75(2):202--211, 1989.

\bibitem{AlgorithmsInInvariantTheory}
B.~Sturmfels.
\newblock {\em Algorithms in invariant theory}.
\newblock Springer-Verlag, Berlin, Heidelberg, 1993.

\bibitem{Vakil}
R.~Vakil.
\newblock {\em The Rising Sea, Foundations of Algebraic Geometry}.
\newblock Available at \url{http://math.stanford.edu/~vakil/216blog/FOAGnov1817public}, 2017.

\bibitem{whiteley1996some}
W.~Whiteley.
\newblock Some matroids from discrete applied geometry.
\newblock {\em Contemporary Mathematics}, 197:171--312, 1996.

\bibitem{whitney1992abstract}
H.~Whitney.
\newblock On the abstract properties of linear dependence.
\newblock {\em Amer. J. Math.}, 57(3):509, jul 1935.

\end{thebibliography}

\medskip
{\footnotesize\noindent {\bf Authors' addresses}
\medskip

\noindent{Emiliano Liwski, 
KU Leuven}\hfill {\tt  emiliano.liwski@kuleuven.be} 
\\ 
\noindent{Fatemeh Mohammadi, 
KU Leuven} \hfill {\tt fatemeh.mohammadi@kuleuven.be}
\\ 
\noindent{Lisa Vandebrouck, 
KU Leuven} \hfill {\tt lisa.vandebrouck@student.kuleuven.be}
}\\
\end{document}